\documentclass[a4paper,11pt]{amsart}
\usepackage[foot]{amsaddr}

\usepackage{epsf,graphicx,epsfig}
\usepackage{amscd}
\usepackage{chngcntr}
\usepackage{enumitem}
\usepackage{amsmath,latexsym,amssymb,amsthm}
\usepackage[nospace,noadjust]{cite}
\usepackage{textcomp}
\usepackage{setspace,cite}
\usepackage{lscape,fancyhdr,fancybox}
\usepackage[all,cmtip]{xy}
\counterwithout{equation}{section}

\usepackage{graphicx}

\usepackage{color}
\usepackage{url}
\usepackage{enumitem}
\usepackage[mathscr]{euscript}
  \theoremstyle{plain}
\newtheorem{theorem}{Theorem}[section]
\newtheorem{corollary}[theorem]{Corollary}
\newtheorem{lemma}[theorem]{Lemma}
\newtheorem{proposition}[theorem]{Proposition}
\newtheorem{example}[theorem]{Example}
\theoremstyle{definition}
\newtheorem{definition}[theorem]{Definition}
\theoremstyle{remark}
\newtheorem{remark}[theorem]{Remark}
\numberwithin{equation}{section}
\makeatletter
\def\@settitle{\begin{center}%
  \baselineskip14\p@\relax
    \normalfont\LARGE%<- NEW
  \@title
  \end{center}%
}

\makeatother
\makeatletter
\renewcommand{\email}[2][]{%
  \ifx\emails\@empty\relax\else{\g@addto@macro\emails{,\space}}\fi%
  \@ifnotempty{#1}{\g@addto@macro\emails{\textrm{(#1)}\space}}%
  \g@addto@macro\emails{#2}%
}
\makeatother

\begin{document}

\title{${\mathcal{O}}$-O{perators} {on Hom-Lie algebras}}
\author[Satyendra Kumar Mishra and Anita Naolekar]{\textbf\normalfont\large S\MakeLowercase{atyendra} K\MakeLowercase{umar} M\MakeLowercase{ishra and }A\MakeLowercase{nita} N\MakeLowercase{aolekar}}
\email{satyamsr10@gmail.com}
\email{anita@isibang.ac.in}
\address{Statistics and Mathematics Unit, Indian Statistical Institute, Bangalore center, India.}
%\author{ 
\footnotetext{AMS Mathematics Subject Classification (2010): $17$B$61,$ $17$A$30,$ $17$B$99,$ $16$T$25$.}
\footnotetext{{\it{Keyword}}: $\mathcal{O}$-operators; Hom-Lie algebras; Hom-pre-Lie algebras; deformation cohomology; differential graded Lie algebra.}
\begin{abstract}
$\mathcal{O}$-operators (also known as relative Rota-Baxter operators) on Lie algebras have several applications in integrable systems and the classical Yang-Baxter equations. In this article, we study $\mathcal{O}$-operators on hom-Lie algebras. We define cochain complex for $\mathcal{O}$-operators on hom-Lie algebras with respect to a representation. Any $\mathcal{O}$-operator induces a hom-pre-Lie algebra structure. We express the cochain complex of an $\mathcal{O}$-operator in terms of certain hom-Lie algebra cochain complex of the sub-adjacent hom-Lie algebra associated with the induced hom-pre-Lie algebra. If the structure maps in a hom-Lie algebra and its representation are invertible, then we can extend the above cochain complex to a deformation complex for $\mathcal{O}$-operators by adding the space of zero cochains. Subsequently, we study linear and formal deformations of $\mathcal{O}$-operators on hom-Lie algebras in terms of the deformation cohomology. In the end, we deduce deformations of $s$-Rota-Baxter operators (of weight 0) and skew-symmetric $r$-matrices on hom-Lie algebras as particular cases of $\mathcal{O}$-operators on hom-Lie algebras.   
\end{abstract}
\maketitle

\section{\normalfont\large\textbf{Introduction}}
%, in the context of some particular deformation called $q$-deformations of Witt and Virasoro algebra of vector fields. In \cite{HALG}, hom-associative algebras are introduced as hom-Lie admissible algebras. Later on, many essential results on hom-Lie algebras and hom-associative algebras followed in \cite{DefHLIE,  HLIE01, QuasiHL1, LGMT18, HALG,  NtHOM, Sheng, UniHLie1, HALG2}. These type of algebras are becoming very popular and several new hom-algebraic structures such as hom-Hopf algebras \cite{HAD09, MSCO10, Yau3}, hom-Jordan algebras \cite{HJRD10} and hom-Poisson algebras \cite{Oliver, hom-Lie, NtHOM} are widely studied. A categorical interpretation of hom-algebra structures is given by S. Caenepeel and I. Goyvaerts in \cite{CGMON}.

%

In this article, we deal with a specific type of algebraic structures, called `hom-algebraic structures'. The defining identities of such an algebraic structure are twisted by endomorphisms. These structures first appeared in the form of hom-Lie algebras \cite{HLIE01}, in the context of deformation of Witt and Virasoro algebras using a more general type of derivations called `$\sigma$-derivations'. The approach of deformations using $\sigma$-derivations yields new deformations of Lie algebras and their central extensions. In particular, $q$-deformations of Witt and Virasoro algebras (For instance, see  \cite{q1,q3,q4,q5,q2,QuasiHL1,q7,q8, q9, q10} for some of such constructions) can be described using $\sigma$-derivations. 

Our main objective is to study $\mathcal{O}$-operators, also known as relative or generalized Rota-Baxter operators on hom-Lie algebras. In 1960, G. Baxter \cite{Baxter} first introduced the notion of Rota–Baxter operators for associative algebras. The Rota-Baxter operators have several applications in probability \cite{Baxter}, combinatorics \cite{Cartier, Guo, Rota}, and quantum field theory \cite{Connes}. In the 1980s, the notion of Rota-Baxter operator of weight $0$ was introduced in terms of the classical Yang-Baxter equation for Lie algebras (see \cite{Survey-Guo} for more details). Later on, B. A. Kupershmidt \cite{Kupershmidt} defined the notion of $\mathcal{O}$-operators as generalized Rota-Baxter operators to understand classical Yang-Baxter equations and related integrable systems.  

Deformation of an algebraic (or analytical) object is a fundamental tool to study the object. In this article, we focus on formal deformation theory of algebraic structures that was first introduced by M. Gerstenhaber for associative algebras in his monumental work \cite{Ger63}-\cite{Ger74}. The deformations of an algebraic structure are controlled by suitable deformation cohomology. For associative algebras, the deformation cohomology is given by the Hochschild cohomology, and in the Lie algebra case, it is the Chevalley-Eilenberg cohomology (see \cite{NR66, NR67a, NR67b}). Likewise, for hom-Lie algebras and hom-associative algebras, there are cohomologies with coefficients in adjoint representations (generalizing Chevalley-Eilenberg and Hochschild cohomology, respectively in \cite{DefHLIE}), playing the role of the required deformation cohomologies. More recently, in \cite{Sheng3}, the authors developed formal deformation theory for $\mathcal{O}$-operators on Lie algebras. They constructed a differential graded Lie algebra (dgLa) and characterized $\mathcal{O}$-operators as Maurer-Cartan elements of this dgLa, which in turn allows them to introduce deformation cohomology for $\mathcal{O}$-operators. 

In this paper, we develop a cohomology for $\mathcal{O}$-operators on hom-Lie algebras. For this purpose, we follow the approach developed in \cite{Sheng3}. In particular, we use the graded Lie algebra structure on the deformation complex of hom-Lie algebras \cite{DefHLIE} and derived bracket construction by Voronov \cite{Voronov} to obtain an explicit graded Lie algebra. We characterize $\mathcal{O}$-operators on hom-Lie algebras as Maurer-Cartan elements of this graded Lie algebra, which allows us to construct a cochain complex defining the cohomology for an $\mathcal{O}$-operator on hom-Lie algebras. We interpret this cohomology as the hom-Lie algebra cohomology of a particular hom-Lie algebra with coefficients in a suitable representation. In the sequel, we show that if the structure maps are invertible, the cochain complex mentioned above extends to a deformation complex for $\mathcal{O}$-operators on hom-Lie algebras.

The section-wise summary of the paper is as follows.
In Section $2$, we recall some basic definitions and results on hom-Lie algebras and their representations. We define a modified complex for regular hom-Lie algebras with coefficients in a representation. We characterize hom-Lie algebras equipped with a representation in terms of Maurer-Cartan elements of a graded Lie algebra (defined in \cite{DefHLIE}).
 
In Section $3$, we define the notion of $\mathcal{O}$-operators on a hom-Lie algebra with respect to a representation. There is a graded Lie algebra structure on the deformation complex of hom-Lie algebra, which we use in order to obtain a graded Lie algebra whose Maurer-Cartan elements are $\mathcal{O}$-operators on hom-Lie algebras with respect to a representation. With this characterization, we get a differential graded Lie algebra associated to an $\mathcal{O}$-operator. Consequently, we define a cohomology for $\mathcal{O}$-operators on hom-Lie algebras with respect to a representation. We show that an $\mathcal{O}$-operator $T:V\rightarrow \mathfrak{g}$ on a hom-Lie algebra $(\mathfrak{g},[~,~],\alpha)$ with respect to a representation $(V,\beta,\rho)$, induces a hom-pre-Lie algebra structure on $V$. Furthermore, the cohomology of the $\mathcal{O}$-operator $T$ can be given in terms of the hom-Lie algebra cohomology of the sub adjacent hom-Lie algebra $(V,[~,~]^c,\beta)$ of the induced pre-Lie algebra with coefficients in a representation $(\mathfrak{g},\alpha,\rho_T)$. In the case when the structure maps $\alpha$ and $\beta$ are invertible, we define an extended cochain complex (including $0$-cochains) $\big(\widetilde{C}^*_{\beta,\alpha}(V,\mathfrak{g}),\delta_{\beta,\alpha}\big)$ for the $\mathcal{O}$-operator $T:V\rightarrow \mathfrak{g}$ on the hom-Lie algebra $(\mathfrak{g},[~,~],\alpha)$ with respect to the representation $(V,\beta,\rho)$.   

In Section $4$, we discuss deformations of $\mathcal{O}$-operators on hom-Lie algebras. For this section, we consider the structure maps to be invertible for both the hom-Lie algebras and their representations. We will see that the modified complex $\big(\widetilde{C}^*_{\beta,\alpha}(V,\mathfrak{g}),\delta_{\beta,\alpha}\big)$ is not possible unless we assume invertibility of the structure maps. This modified complex gives us deformation cohomology of $\mathcal{O}$-operators. First, we consider linear deformations of $\mathcal{O}$-operators. We discuss trivial linear deformations in terms of Nijenhuis elements. Then we consider formal deformations of $\mathcal{O}$-operators. We show that two equivalent formal deformations of $\mathcal{O}$-operators have cohomologous infinitesimals. We also consider the problem of extending a finite order deformation to the next higher-order deformation. 

In the last section, we consider $s$-Rota-Baxter operators of weight $0$ and skew-symmetric $r$-matrices on hom-Lie algebras as particular cases of $\mathcal{O}$-operators. Consequently, we discuss deformations of $s$-Rota-Baxter operators of weight $0$ and skew-symmetric $r$-matrices on hom-Lie algebras.

\section{\normalfont\large\textbf{{Preliminaries}}}
Throughout the paper, we consider the base field to be $k$ with characteristic zero. All linear maps and tensor products are taken over the field $k$ unless otherwise stated.

\subsection{Hom-Lie algebras}
 Let us recall some definitions and results related to hom-Lie algebras and their representations. 

\begin{definition}
 A (multiplicative) hom-Lie algebra is a triplet $(\mathfrak{g},[~,~],\alpha)$, where $\mathfrak{g}$ is a vector space equipped with a skew-symmetric bilinear map $[~,~]:\mathfrak{g}\otimes \mathfrak{g}\rightarrow \mathfrak{g}$, and a linear map $\alpha:\mathfrak{g}\rightarrow \mathfrak{g}$ satisfying $\alpha[x,y]=[\alpha(x),\alpha(y)]$ such that
\begin{equation}
[\alpha(x),[y,z]]+[\alpha(y),[z,x]]+[\alpha(z),[x,y]]=0, ~~~~\mbox{for all}~~x,y,z\in \mathfrak{g}.
\end{equation}
Furthermore, if $\alpha:\mathfrak{g}\rightarrow \mathfrak{g}$ is a vector space automorphism of $\mathfrak{g}$, then the hom-Lie algebra $(\mathfrak{g},[~,~],\alpha)$ is called a regular hom-Lie algebra.
\end{definition}

\begin{example}
Given a Lie algebra $\mathfrak{g}$ with a Lie algebra homomorphism $\alpha:\mathfrak{g}\rightarrow \mathfrak{g}$, we can define a hom-Lie algebra as the triplet $(\mathfrak{g},\alpha\circ [~,~],\alpha)$, where $[~,~]$ is the underlying Lie bracket. 
\end{example}

\begin{definition}\label{Rep-hom-Lie}(\cite{Sheng}) A representation of a hom-Lie algebra $(\mathfrak{g},[~,~],\alpha)$ on a $k$-vector space $V$ with respect to $\beta\in\mathsf{End}(V)$ is a linear map $\rho: \mathfrak{g}\rightarrow \mathsf{End}(V)$ such that 
\begin{equation}
\rho(\alpha(x))(\beta(v))= \beta(\rho(x)(v)),
\end{equation}
\begin{equation}
\rho([x,y])(\beta(v))= \rho(\alpha(x))\rho(y)(v)-\rho(\alpha(y))\rho(x)(v),
\end{equation}
for all $x,y \in \mathfrak{g}$ and $v\in V$. 
%A representation $(V,\beta,\rho)$ is said to be 'regular' if the map $\beta:V\rightarrow V$ is a vector space automorphism of $V$.
\end{definition}
Let us denote a representation $\rho$ on $V$ with respect to $\beta\in \mathsf{End}(V)$ by a triple $(V,\beta,\rho)$. Throughout the paper, for simplification, we use the notation
$$\{x,v\}:=\rho(x)(v), \quad\mbox{for all } x\in \mathfrak{g}, ~v\in V.$$

\begin{example}[\cite{Sheng}]\label{adjoint representation}
 For any integer $s\geq 0$, we can define the $\alpha^s$-adjoint representation of a hom-Lie algebra $(\mathfrak{g},[~,~],\alpha)$ on $\mathfrak{g}$ as follows
$$\mathsf{ad}_x^s(y):=[\alpha^s(x),y]\quad\mbox{for all}~x,y\in \mathfrak{g}.$$
Let us denote the $\alpha^s$-adjoint representation of the hom-Lie algebra $(\mathfrak{g},[~,~],\alpha)$ by the triple $(\mathfrak{g},\alpha,\mathsf{ad}^s)$. Also, we denote $\mathsf{ad}_x^0$ simply by $\mathsf{ad}_x$ for any $x\in\mathfrak{g}$.
\end{example}

\begin{example}[\cite{hom-Liebi}]\label{coadjoint rep}
Let $(\mathfrak{g},[~,~],\alpha)$ be a regular hom-Lie algebra and $(V,\beta,\rho)$ be a hom-Lie algebra representation with $\beta$ an invertible linear map. Let us define a map $\rho^\star:\mathfrak{g}\rightarrow\mathsf{End}(V^*)$ by
\begin{align*}
\langle\rho^{\star}(x)(\xi),v\rangle :=&\langle\rho^*(\alpha(x))((\beta^{-2})^*(\xi)),v\rangle \\
=&\langle\xi,\rho(\alpha^{-1}(x))(\beta^{-2}(v))\rangle,
\end{align*}
for all $x\in\mathfrak{g}$ and $v\in V$. Then the triplet $(V^*,(\beta^{-1})^*,\rho^\star)$ is a representation of the hom-Lie algebra $(\mathfrak{g},[~,~],\alpha)$ on the dual vector space $V^*$ with respect to the map $(\beta^{-1})^*$. This is also known as the `dual representation' to $(V,\beta,\rho)$. 

In particular, let us also recall that the `coadjoint representation' of a regular hom-Lie algebra $(\mathfrak{g},[~,~],\alpha)$ on $\mathfrak{g}^*$ with respect to the map $(\alpha^{-1})^*$ is given by the triplet $(\mathfrak{g}^*,(\alpha^{-1})^*,\mathsf{ad}^\star)$, where
$$\langle \mathsf{ad}^\star(x)(\xi),y\rangle=\langle\xi, [\alpha^{-1}(x),\alpha^{-2}(y)]\rangle, \quad\mbox{for all }x,y\in\mathfrak{g},~\xi\in \mathfrak{g}^*.$$
\end{example}

\begin{lemma}\label{direct sum}
Let $(\mathfrak{g},[~,~],\alpha)$ be a hom-Lie algebra and $V$ be a vector space equipped with a linear map $\beta\in \mathsf{End}(V)$ and an action $\rho:\mathfrak{g}\otimes V\rightarrow V$. Then, the triplet $(V,\beta,\rho)$ is a hom-Lie algebra representation on $V$ if and only if $(\mathfrak{g}\oplus V,[~,~]_{\rho},\alpha+\beta)$ is a hom-Lie algebra, where the bracket $[~,~]_{\rho}$ and the map $\alpha+\beta$ are given by 
\begin{align*}
[x+v,y+w]_{\rho}&=[x,y]+\rho(x)(w)-\rho(y)(v),\\
(\alpha+\beta)(x,v)&=(\alpha(x),\beta(v)),\quad\mbox{for all }x,y\in \mathfrak{g}~\mbox{and } v,w\in V.
\end{align*}
\end{lemma}
\begin{proof}
The proof of the lemma is straightforward (see \cite{Sheng} for details).
\end{proof}
The hom-Lie algebra $(\mathfrak{g}\oplus V,[~,~]_{\rho},\alpha+\beta)$ is called the semi-direct product hom-Lie algebra for a hom-Lie algebra $(\mathfrak{g},[~,~],\alpha)$ with a representation $(V,\beta,\rho)$.
\subsection{A cochain complex for hom-Lie algebras}
Let $(\mathfrak{g},[~,~],\alpha)$ be a hom-Lie algebra with a representation $(V, \beta, \rho)$ on a vector space $V$. We define a cochain complex $(C^*_{\alpha,\beta}(\mathfrak{g},V),\delta_{\alpha,\beta})$ for the hom-Lie algebra $(\mathfrak{g},[~,~],\alpha)$ with coefficients in the representation $(V, \beta, \rho)$. Here, 
$$C^*_{\alpha,\beta}(\mathfrak{g},V):=\bigoplus_{n\geq 1}C^n_{\alpha,\beta}(\mathfrak{g},V)$$
and $C^n_{\alpha,\beta}(\mathfrak{g},V)$ is a subspace of $\mathrm{Hom}(\wedge^n \mathfrak{g},V)$ consisting of all those linear maps $f:\wedge^n \mathfrak{g}\rightarrow V,$ which satisfy the following condition
$$f(\alpha(x_1),\cdots,\alpha(x_n))=\beta(f(x_1,x_2,\cdots,x_n)).$$
The coboundary map $\delta_{\alpha,\beta}:C^n_{\alpha,\beta}(\mathfrak{g},V)\rightarrow C^{n+1}_{\alpha,\beta}(\mathfrak{g},V)$ is given by 
\begin{equation}\label{coboundary:hom-Lie algebra}
\begin{split}
\delta_{\alpha,\beta} f(x_1,\cdots,x_{n+1}):= &\sum_{i=1}^{n+1}(-1)^{i+1}\rho(\alpha^{n-1}(x_i))(f(x_1,\cdots,\hat{x_i},\cdots,x_{n+1}))\\
&+\sum_{i<j} (-1)^{i+j}f([x_i,x_j],\alpha(x_1),\cdots,\hat{\alpha(x_i)},\cdots,\hat{\alpha(x_j)},\cdots,\alpha(x_{n+1}))
\end{split}
\end{equation}
for all $f\in C^n_{\alpha,\beta}(\mathfrak{g},V)$ and $x_1,x_2,\cdots x_{n+1}\in \mathfrak{g}$. Let us denote by $H^*_{\alpha,\beta}(\mathfrak{g},V)$, the cohomology space  associated to the cochain complex $(C^*_{\alpha,\beta}(\mathfrak{g},V),\delta_{\alpha,\beta})$. 

\begin{remark}
This cochain complex is different from the one defined in \cite{Sheng}. If $V=\mathfrak{g},$ $\beta=\alpha,$ and the action $\rho: \mathfrak{g}\otimes V\rightarrow V$ is given by the underlying hom-Lie bracket, we denote the above cochain complex by $(C^*_{\alpha}(\mathfrak{g},\mathfrak{g}),\delta_{\alpha})$. This complex is the same as the deformation complex of the hom-Lie algebra $(\mathfrak{g},[~,~],\alpha)$ (defined in \cite{DefHLIE}). The cohomology of the complex $(C^*_{\alpha}(\mathfrak{g},\mathfrak{g}),\delta_{\alpha})$ serves as deformation cohomology for hom-Lie algebra.
\end{remark}

\begin{remark}\label{regular hom-Lie algebra cohomology}
Let $(\mathfrak{g},[~,~],\alpha)$ be a regular hom-Lie algebra and $(V, \beta, \rho)$ be a hom-Lie algebra representation, where $\beta:V\rightarrow V$ is invertible. Then, one can define 
$$C^0_{\alpha,\beta}(\mathfrak{g},V):=\{v\in V~| ~\beta(v)=v\},$$
and the coboundary $\delta_{\alpha,\beta}:C^0_{\alpha,\beta}(\mathfrak{g},V)\rightarrow C^1_{\alpha,\beta}(\mathfrak{g},V)$ on $0$-cochains is given by
$$\delta_{\alpha,\beta}(v)(x):=\rho(\alpha^{-1}(x))(v),~\quad\mbox{for all }v\in C^0_{\alpha,\beta}(\mathfrak{g},V) ~\mbox{and}~ x\in \mathfrak{g}.$$
It follows that we have a modified cochain complex $\big(\widetilde{C}^*_{\alpha,\beta}(\mathfrak{g},V),\delta_{\alpha,\beta}\big)$, where 
$$\widetilde{C}^*_{\alpha,\beta}(\mathfrak{g},V):=\bigoplus_{n\geq 0} C^n_{\alpha,\beta}(\mathfrak{g},V).$$
We denote the associated cohomology by $\widetilde{H}^*_{\alpha,\beta}(\mathfrak{g},V)$.
\end{remark}

\subsection{Hom-Lie algebras and their representations in terms of Maurer-Cartan elements}
Let $\mathfrak{g}$ be a vector space equipped with a linear map $\alpha:\mathfrak{g}\rightarrow\mathfrak{g}$. Let us consider the graded vector space $C^*_{\alpha}(\mathfrak{g};\mathfrak{g})$. Then, we recall from  \cite{DefHLIE} that for any $\varphi\in C^p_{\alpha}(\mathfrak{g},\mathfrak{g})$ and $\psi\in C^q_{\alpha}(\mathfrak{g},\mathfrak{g})$, a circle product $\varphi\circ_\alpha \psi$ is defined by the following expression
\begin{equation}\label{Defofcirc}
\begin{split}
&(\varphi \circ_{\alpha} \psi)(x_1, x_2, \ldots, x_{p+q+1})\\ &= \textstyle{\sum\limits_{\tau \in Sh (q+1, p)} (-1)^{|\tau|}
\varphi \big(\psi (x_{\tau(1)}, \ldots, x_{\tau (q+1)}),\alpha^q(x_{\tau (q+2)}) , \ldots, \alpha^q(x_{\tau (p+q+1))}) \big)},
\end{split}
\end{equation}
where $Sh (q+1, p)$ denotes the set of $(q+1,p)$ shuffles in $S_{q+p+1}$ (the symmetric group on the set $\{1,2,\cdots,q+p+1\}$). For any permutation $\tau\in S_{q+p+1}$, the notation $|\tau|$ denotes the signature of the permutation $\tau$. 

We can define a bracket of degree $-1$ on the graded vector space $C^*_{\alpha}(\mathfrak{g},\mathfrak{g})$ in terms of the circle product
\begin{equation}\label{GLB}
[\varphi, \psi]_N^{\alpha} := (-1)^{pq} \varphi \circ_{\alpha} \psi - \psi \circ_{\alpha} \varphi.
\end{equation}

Thus, by a degree shift we obtain a graded Lie algebra structure on $C^{*-1}_{\alpha}(\mathfrak{g},\mathfrak{g})$. Hom-Lie algebra structures on $(\mathfrak{g},\alpha)$  corresponds bijectively to Maurer-Cartan elements of the graded Lie algebra $C^{*-1}_{\alpha}(\mathfrak{g},\mathfrak{g})$, i.e., the elements $\mu\in C^{2}_{\alpha}(\mathfrak{g},\mathfrak{g})$ satisfying $[\mu,\mu]_N^{\alpha}=0$. If a hom-Lie algebra structure $(\mathfrak{g},[~,~],\alpha)$ corresponds to such an element $\mu\in C^{2}_{\alpha}(\mathfrak{g},\mathfrak{g})$, one obtains a differential graded Lie algebra structure on $C^{*-1}_{\alpha}(\mathfrak{g},\mathfrak{g})$ with the differential $d_{\mu}=[\mu,-]_N^{\alpha}$. This differential $d_{\mu}$ coincides with the coboundary operator $\delta_{\alpha}$ given by equation \eqref{coboundary:hom-Lie algebra}.

Let $(\mathfrak{g},\alpha)$ and $(V,\beta)$ be vector spaces equipped with linear operators. We now consider the graded Lie algebra 
$$\textstyle{\big(\mathcal{G}^*:=C^{*-1}_{\alpha}(\mathfrak{g}\oplus V,\mathfrak{g}\oplus V),~[~,~]_N^{\alpha+\beta}\big)}$$ 
associated to the pair $(\mathfrak{g}\oplus V,\alpha+\beta)$. Let $\mu: \wedge^2(\mathfrak{g})\rightarrow\mathfrak{g}$ and $\rho:\mathfrak{g}\rightarrow \mathsf{End}(V)$ be linear maps. Then let us define a linear map $\mu+\rho:\wedge^2(\mathfrak{g}\oplus V)\rightarrow \mathfrak{g}\oplus V$ by
$$\mu+\rho(x+v,y+w)=\mu(x,y)+\rho(x)(w)-\rho(y)(v),\quad \mbox{for any }x,y\in \mathfrak{g}~~\mbox{and }v,w\in V.$$
With the above notations, we have the following proposition.

\begin{proposition}\label{Maurer-Cartan}
The map $\mu$ defines a hom-Lie algebra structure on the pair $(\mathfrak{g},\alpha)$ and the map $\rho$ defines a hom-Lie algebra representation on the pair $(V,\beta)$ if and only if $\mu+\rho$ is a Maurer-Cartan element of the graded Lie algebra $\textstyle{\big(\mathcal{G}^*,~[~,~]_N^{\alpha+\beta}\big)}$.
\end{proposition}
\begin{proof}
First, let us observe that the map $\mu+\rho\in \mathcal{G}^1$ if and only if 
$$\mu+\rho\big(\alpha+\beta(x+v),\alpha+\beta(y+w)\big)=(\alpha+\beta)\big(\mu+\rho(x+v,y+w)\big),$$
which is equivalent to the following expressions
$$\mu(\alpha(x),\alpha(y))=\alpha\big(\mu(x,y)\big)\quad \mbox{and}\quad \rho(\alpha(x))(\beta(v))=\beta(\rho(x)(v)),\quad \mbox{for all }x,y\in \mathfrak{g},~v\in V.$$
Moreover, the map $\mu+\rho$ is a Maurer-Cartan element if and only if 
\begin{equation}\label{square-zero}
\textstyle{[\mu+\rho,\mu+\rho]_N^{\alpha+\beta}(x+u,y+v,z+w)=-2~\big((\mu+\rho)\circ_{\alpha+\beta}(\mu+\rho)\big)(x+u,y+v,z+w)=0},
\end{equation}
or equivalently
\begin{align}\label{HJacobi n representation}
\nonumber
[[x,y],\alpha(z)]+[[y,z],\alpha(x)]+[[z,x],\alpha(y)]&=0,\quad \mbox{for all }x,y,z \in \mathfrak{g}\\
\rho([x,y])(\beta(w))- \rho(\alpha(x))\rho(y)(w)+\rho(\alpha(y))\rho(x)(w)&=0,\quad \mbox{for all }x,y\in \mathfrak{g}~~\mbox{and }w\in V.
\end{align} 
Note that we obtain the equation \eqref{HJacobi n representation} from equation \eqref{square-zero} by taking $u=v=0$. Hence, the proposition follows.
\end{proof}

\section{\normalfont\large\textbf{{$\mathcal{O}$-operators on hom-Lie algebras}}}

In this section, we define the notion of $s$-Rota-Baxter operators and $\mathcal{O}$-operators on hom-Lie algebras. We use the graded Lie algebra structure on the deformation complex of a hom-Lie algebra and derived bracket construction to define a graded Lie algebra whose Maurer-Cartan elements are precisely the $\mathcal{O}$-operators on hom-Lie algebras. Subsequently, we obtain a differential graded Lie algebra associated to an $\mathcal{O}$-operator.

\begin{definition}\label{def:Rota-Baxter operators}
Let $(\mathfrak{g},[~,~],\alpha)$ be a hom-Lie algebra and $s$ be a non-negative integer. Then, a linear operator $\mathcal{R}: \mathfrak{g} \rightarrow \mathfrak{g}$ is called an $s$-Rota-Baxter operator of weight $\lambda$ on $(\mathfrak{g},[~,~],\alpha)$ if $\mathcal{R}\circ \alpha= \alpha\circ \mathcal{R}$ and the following identity is satisfied
\begin{equation*} 
\quad\quad\quad\quad\quad[\mathcal{R}(x), \mathcal{R}(y)]= \mathcal{R}([\alpha^s \mathcal{R}(x), y]+ [x, \alpha^s \mathcal{R}(y)]+\lambda [x,y]),\quad\mbox{for all }x,y\in \mathfrak{g}.
\end{equation*}
\end{definition}

For $\alpha=\mathsf{Id}$, the Definition \ref{def:Rota-Baxter operators} coincides with the notion of Rota-Baxter operators on a Lie algebra. 

\begin{definition}\label{O-operator}
Let $(\mathfrak{g},[~,~],\alpha)$ be a hom-Lie algebra and $(V,\beta,\rho)$ be a hom-Lie algebra representation. A linear map $T:V\rightarrow \mathfrak{g}$ is called an $\mathcal{O}$-operator on $(\mathfrak{g},[~,~],\alpha)$ with respect to the representation $(V,\beta,\rho)$ if the following conditions hold
\begin{align*}
T\circ\beta&=\alpha\circ T,\\
[Tu,Tv]&=T\big(\{Tu,v\}-\{Tv,u\}\big),\quad\mbox{ for all }  u,v\in V.
\end{align*}
\end{definition}

\begin{remark}\label{5.2}
Let us recall from Example \ref{adjoint representation}, the $\alpha^s$-adjoint representation $(\mathfrak{g}, \alpha, \mathsf{ad}^s)$ of a hom-Lie algebra $(\mathfrak{g},[~,~],\alpha)$ for any integer $s\geq 0$. Then, an $s$-Rota-Baxter operator of weight $0$ on the hom-Lie algebra $(\mathfrak{g},[~,~],\alpha)$ is an $\mathcal{O}$-operator on $(\mathfrak{g},[~,~],\alpha)$ with respect to the representation $(\mathfrak{g}, \alpha, \mathsf{ad}^s)$. Thus, the notion of $\mathcal{O}$-operators is a generalization of Rota-Baxter operators and therefore also known as relative or generalized Rota-Baxter operators.
\end{remark}

\begin{example}
If $\alpha=\mathsf{Id}_{\mathfrak{g}}$ and $\beta=\mathsf{Id}_V$, then the Definition \ref{O-operator} coincides with the notion of $\mathcal{O}$-operators on a Lie algebra. 
\end{example}

\begin{example}
Let $T:V\rightarrow \mathfrak{g}$ be an $\mathcal{O}$-operator on a Lie algebra $(\mathfrak{g},[~,~])$ with respect to a Lie algebra representation $\rho$ on $V$. A pair $(\phi_{\mathfrak{g}},\phi_V)$ is an endomorphism of the $\mathcal{O}$-operator $T$ if 
\begin{align*}
T\circ \phi_{V}&=\phi_{\mathfrak{g}}\circ T\quad \mbox{and}\\
\rho(\phi_{\mathfrak{g}}(x))(\phi_{V}(v))&=\phi_V(\rho(x)(v)), \quad \mbox{for all } x\in \mathfrak{g}, ~v\in V. 
\end{align*}

 Let us consider the hom-Lie algebra $(\mathfrak{g},[~,~]_{\phi_\mathfrak{g}},\phi_{\mathfrak{g}})$ obtained by composition, where the hom-Lie bracket is given by 
  $$[~,~]_{\phi_\mathfrak{g}}:=\phi_{\mathfrak{g}}\circ [~,~].$$ 
If we consider the composition $\rho_{\phi_V}:=\phi_V\circ \rho$, then the triplet $(V,\phi_V,\rho_{\phi_V})$ is a hom-Lie algebra representation of $(\mathfrak{g},[~,~]_{\phi_\mathfrak{g}},\phi_{\mathfrak{g}})$. Moreover, 
$$[T(v),T(w)]_{\phi_\mathfrak{g}}=\phi_{\mathfrak{g}}[T(v),T(w)]=\phi_{\mathfrak{g}}\big(T(\rho(T(v))(w)-\rho(T(w))(v)\big),$$
and 
$$T\big(\rho_{\phi_V}(T(v))(w)-\rho_{\phi_V}(T(w))(v)\big)=T\big(\phi_V(\rho(T(v))(w)-\rho(T(w))(v)\big), $$
for all $v,w \in V$. Clearly, it follows that the map $T:V\rightarrow \mathfrak{g}$ is an $\mathcal{O}$-operator on hom-Lie algebra $(\mathfrak{g},[~,~]_{\phi_\mathfrak{g}},\phi_{\mathfrak{g}})$ with respect to the hom-Lie algebra representation $(V,\phi_V,\rho_{\phi_V})$.
\end{example}

The following proposition gives a characterization of an $\mathcal{O}$-operator $T$ in terms of a hom-Lie subalgebra structure on the graph of $T$.

\begin{proposition}
A map $T:V\rightarrow \mathfrak{g}$ is an $\mathcal{O}$-operator on $(\mathfrak{g},[~,~],\alpha)$ with respect to the representation $(V,\beta,\rho)$ if and only if the graph of the map $T$ $$\mathrm{Gr}(T)=\{(T(v),v)|~v\in V\}$$
is a  hom-Lie subalgebra of the semi-direct product hom-Lie algebra $(\mathfrak{g}\oplus V,[~,~]_{\rho},\alpha+\beta)$, defined in Lemma \ref{direct sum}.
\end{proposition}

\begin{proof}
The proposition follows by a straightforward calculation, we leave it to the reader.
\end{proof}

It is well-known that $\mathcal{O}$-operators on Lie algebras can be characterized in terms of the Nijenhuis operators. In the next result, we characterize $\mathcal{O}$-operators on hom-Lie algebras in terms of the Nijenhuis operators. Let us first recall from \cite{Sheng} that a linear map $N:\mathfrak{g}\rightarrow \mathfrak{g}$ is called a Nijenhuis operator on the hom-Lie algebra $(\mathfrak{g},[~,~],\alpha)$ if
$$[N(x),N(y)]=N\big([N(x),y]-[N(y),x]-N([x,y])\big)\quad \mbox{for all }x,y\in \mathfrak{g}.$$
Then, we have the following characterization of $\mathcal{O}$-operators on hom-Lie algebras.
\begin{proposition}
A map $T:V\rightarrow \mathfrak{g}$ is an $\mathcal{O}$-operator on $(\mathfrak{g},[~,~],\alpha)$ with respect to the representation $(V,\beta,\rho)$ if and only if the operator 
$$N_T=\begin{bmatrix}
   0 & T \\
    0  & 0
\end{bmatrix}: \mathfrak{g}\oplus V\rightarrow \mathfrak{g}\oplus V $$  
is a Nijenhuis operator on the semi-direct product hom-Lie algebra $(\mathfrak{g}\oplus V,[~,~]_{\rho},\alpha+\beta)$.
\end{proposition}
\begin{proof} Let us consider the following expressions, where we use definition of the map $N_T$ and the bracket $[~,~]_{\rho}$. 
\begin{equation}\label{Char2: eq1}
[N_T(x+v),N_T(y+w)]_{\rho}= [T(v)+0,T(w)+0]_{\rho}=[T(v),T(w)], 
\end{equation}
and
\begin{align}\label{Char2: eq2}
\nonumber
&N_T\big([N_T(x+v),y+w]_{\rho}-[N_T(y+w),x+v]_{\rho}-N_T([x+v,y+w]_{\rho})\big)\\\nonumber
=&N_T \big(([T(v),y]+\{T(v),w\})-([T(w),x]+\{T(w),v\})-(0+T(\{x,w\}-\{y,v\}))\big)\\
=& T(\{T(v),w\}-\{T(w),v\}),
\end{align}
for all $x,y\in\mathfrak{g},$ $v,w\in V$. By the above equations \eqref{Char2: eq1}-\eqref{Char2: eq2}, it is clear that the condition  
$$[N_T(x+v),N_T(y+w)]_{\rho}=N_T\big([N_T(x+v),y+w]_{\rho}-[N_T(y+w),x+v]_{\rho}-N_T([x+v,y+w]_{\rho})\big)$$
is equivalent to the condition
$$[T(v),T(w)]=T(\{T(v),w\}-\{T(w),v\}),$$
for all $x,y\in\mathfrak{g},$ $v,w\in V$. 
\end{proof}
\begin{definition}\label{morphism of O-operators}
Let $T:V\rightarrow \mathfrak{g}$ and $T^{\prime}:V\rightarrow \mathfrak{g}$ be two $\mathcal{O}$-operators on the hom-Lie algebra $(\mathfrak{g},[~,~],\alpha)$ with respect to the representation $(V,\beta,\rho)$. A homomorphism from $T$ to $T^{\prime}$ is given by a pair $(\phi_\mathfrak{g},\phi_V)$, consisting of a hom-Lie algebra homomorphism $\phi_{\mathfrak{g}}:\mathfrak{g}\rightarrow \mathfrak{g}$ and a linear map $\phi_V: V\rightarrow V$ such that following conditions are satisfied
\begin{equation}\label{morphism:cond1}
T^{\prime}\circ \phi_V =\phi_{\mathfrak{g}}\circ T,\quad\quad\quad
\end{equation}
\begin{equation}\label{morphism:cond2}
\phi_V\circ \beta  =\beta\circ \phi_V,\quad\quad\quad
\end{equation}
\begin{equation}\label{morphism:cond3}
\quad\quad \quad\quad\quad\quad~~~~~~~\phi_V(\rho(x)(v)) =\rho(\phi_g(x))(\phi_V(v)), \quad\mbox{for all  }v\in V.
\end{equation}
\end{definition}

\subsection{A differential graded Lie algebra}

Let $(\mathfrak{g},[~,~],\alpha)$ be a hom-Lie algebra with a representation $(V,\beta,\rho)$. Then, we have a graded Lie algebra 
$$\textstyle{\big(\mathcal{G}^*:=C^{*-1}_{\alpha}(\mathfrak{g}\oplus V,\mathfrak{g}\oplus V),~[~,~]_N^{\alpha+\beta}\big)}$$ 
associated to the pair $(\mathfrak{g}\oplus V,\alpha+\beta)$.
 Then, by Proposition \ref{Maurer-Cartan}, the hom-Lie bracket and the representation correspond to an element $\mu+\rho\in \mathcal{G}^1$ satisfying $[\mu+\rho,\mu+\rho]_N^{\alpha+\beta}=0$. Let us define a map $d_{\mu+\rho}:\mathcal{G}^*\rightarrow \mathcal{G}^{*+1}$ by
$$d_{\mu+\rho}:=[\mu+\rho,-]_N^{\alpha+\beta}.$$
By the graded Jacobi identity for the bracket $[~,~]_N^{\alpha+\beta}$, it is clear that $(\mathcal{G}^*,[~,~]_N^{\alpha+\beta},d_{\mu+\rho})$ is a differential graded Lie algebra.

Now, we define a graded vector space 
\begin{equation*}
C^*_{\beta,\alpha}(V,\mathfrak{g})=\bigoplus_{n\geq 1}C^n_{\beta,\alpha}(V,\mathfrak{g}),
\end{equation*}
 where for each $n\geq 1$, the vector space $C^n_{\beta,\alpha}(V,\mathfrak{g})$ is a subspace of $\mathrm{Hom}(\wedge^n V,\mathfrak{g})$, which consists of all the linear maps $P:\wedge^n V\rightarrow \mathfrak{g}$ satisfying 
$$\alpha(P(v_1,v_2,\ldots,v_{n}))=P(\beta(v_1),\beta(v_2),\ldots,\beta(v_{n})), \quad\mbox{for all }~~ v_1,v_2,\cdots, v_{n}\in V.$$

Let us define a graded Lie bracket 
$$\{\!\!\{ -,- \}\!\!\}:C^n_{\beta,\alpha}(V,\mathfrak{g})\otimes C^m_{\beta,\alpha}(V,\mathfrak{g}) \rightarrow C^{n+m}_{\beta,\alpha}(V,\mathfrak{g})$$
as follows
\begin{equation}\label{definition of bracket}
\{\!\!\{P,Q\}\!\!\}:=(-1)^n[[~\mu+\rho~,~ P]_N^{\alpha+\beta},~Q]_N^{\alpha+\beta}
\end{equation}

By definition of the graded Lie bracket $[~,~]_N^{\alpha+\beta},$ the bracket $\{\!\!\{-,-\}\!\!\}$ can be written as follows:
\begin{align}\label{derived bracket}
\nonumber
&\{\!\!\{P,Q\}\!\!\}(v_1,v_2,\ldots,v_{n+m})\\\nonumber
&=\sum_{\tau\in S_{m,1,n-1}}(-1)^{|\tau|}~ P\Big(\{Q(v_{\tau(1)},\ldots,v_{\tau(m)}),\beta^{m-1}(u_{\tau(m+1)})\},\beta^{m}(v_{\tau(n+2)}),\ldots,\beta^{m}(v_{\tau(n+m)})\Big)\\\nonumber
&+(-1)^{mn}\Bigg(\sum_{\tau\in S_{n,m}}(-1)^{|\tau|}~\big[\alpha^{m-1}P(v_{\tau(1)},\ldots,v_{\tau(n)}),\alpha^{n-1}Q(v_{\tau(n+1)},\ldots,v_{\tau(n+m)})\big]\\
&-\sum_{\tau\in S_{n,1,m-1}}(-1)^{|\tau|}~ Q\Big(\{P(v_{\tau(1)},\ldots,v_{\tau(n)}),\beta^{n-1}(u_{\tau(n+1)})\},\beta^{n}(v_{\tau(n+2)}),\ldots,\beta^{n}(v_{\tau(n+m)})\Big)\Bigg)
\end{align}

Here, $S_{r_1,r_2,\ldots,r_i}$ denotes a $(r_1,r_2,\ldots,r_i)$-shuffle in the permutation group $S_{r_1+r_2+\cdots+r_i}$. The above graded Lie bracket is obtained via the derived bracket construction, introduced by Voronov in \cite{Voronov}. Moreover, for any $T\in C^1_{\beta,\alpha}(V,\mathfrak{g})$, i.e. $T:V\rightarrow \mathfrak{g}$ is a linear map satisfying $T\circ \beta=\alpha\circ T$, we have $$\{\!\!\{T,T\}\!\!\}(v_1,v_2)=2\Big(T\{Tv_1,v_2\}-T\{Tv_1,v_2\}-[Tv_1,Tv_2]\Big),\quad \mbox{for all~~} v_1, v_2\in V.$$ 
In turn, it follows that $\{\!\!\{T,T\}\!\!\}=0$ if and only if $T:V\rightarrow \mathfrak{g}$ is an $\mathcal{O}$-operator on hom-Lie algebra $(\mathfrak{g},[~,~],\alpha)$ with respect to the representation $(V,\beta,\rho)$. 
Thus, we have the following theorem generalizing the Lie algebra case \cite{Sheng3}.

\begin{theorem}\label{Maurer Cartan element}
The graded vector space $C^*_{\beta,\alpha}(V,\mathfrak{g})$ forms a graded Lie algebra with the graded Lie bracket $\{\!\!\{-,-\}\!\!\}$. A linear map $T:V\rightarrow \mathfrak{g}$ satisfying $T\circ \beta=\alpha\circ T$ is an $\mathcal{O}$-operator on hom-Lie algebra $(\mathfrak{g},[~,~],\alpha)$ with respect to the representation $(V,\beta,\rho)$ if and only if $T\in C^1_{\beta,\alpha}(V,\mathfrak{g})$ is a Maurer-Cartan element of the graded Lie algebra $(C^*_{\beta,\alpha}(V,\mathfrak{g}),\{\!\!\{-,-\}\!\!\})$.
\end{theorem}

\begin{remark}\label{dgla}
Let $T:V\rightarrow \mathfrak{g}$ be an $\mathcal{O}$-operator on hom-Lie algebra $(\mathfrak{g},[~,~],\alpha)$ with respect to the representation $(V,\beta,\rho)$. From Theorem \ref{Maurer Cartan element}, $T\in C^1_{\beta,\alpha}(V,\mathfrak{g})$ is a Maurer-Cartan element of the graded Lie algebra $(C^*_{\beta,\alpha}(V,\mathfrak{g}),\{\!\!\{-,-\}\!\!\})$. Then, the $\mathcal{O}$-operator $T$ induces a differential $\delta_T:=\{\!\!\{T,-\}\!\!\}$ on the graded Lie algebra $(C^*_{\beta,\alpha}(V,\mathfrak{g}),\{\!\!\{-,-\}\!\!\})$, which makes it a differential graded Lie algebra.  
\end{remark}

By remark \ref{dgla}, we associate a cochain complex $(C^*_{\beta,\alpha}(V,\mathfrak{g}),\delta_T)$ to an $\mathcal{O}$-operator $T$ on a hom-Lie algebra $(\mathfrak{g},[~,~],\alpha)$ with respect to the representation $(V,\beta,\rho)$. The cohomology of this cochain complex is called the cohomology of the $\mathcal{O}$-operator $T$.

\subsection{Cohomology of $\mathcal{O}$-operators in terms of hom-Lie algebra cohomology}
Now, we describe the cohomology of an $\mathcal{O}$-operator on hom-Lie algebras in terms of hom-Lie algebra cohomology of certain hom-Lie algebra with coefficients in a representation. 

%For this purpose, we first show that any $\mathcal{O}$-operator on a hom-Lie algebra $(\mathfrak{g},[~,~],\alpha)$ with respect to representation $(V,\beta,\rho)$ induces a hom-pre-Lie algebra structure $(V,\cdot_T,\beta)$. We denote by $V^c_\beta$, the sub-adjacent hom-Lie algebra of the induced hom-pre-Lie algebra. We prove that the operator $T$ also induces a representation $\rho_T$ of hom-Lie algebra $V^c_{\beta}$ on $\mathfrak{g}$ with respect to the map $\alpha$. Finally, we show that the complex $(C^*_{\beta,\alpha}(V,\mathfrak{g}),\delta_T)$ coincides with the cochain complex of hom-Lie algebra with coefficients in the representation $(\mathfrak{g},\alpha,\rho_T)$.   

Let us recall that a hom-pre-Lie algebra is a triplet $(V,\cdot,\beta)$, where $V$ is a vector space equipped with a bilinear map $\cdot: V\otimes V\rightarrow V$ and a linear map $\beta: V\rightarrow V$ such that 
\begin{align*}
\beta(u\cdot v)&=\beta(u)\cdot \beta(v), \quad \mbox{and}\\
(u\cdot v)\cdot  \beta(w) - \beta(u)\cdot  (v\cdot  w) &= (v\cdot u)\cdot \beta(w) - \beta(v)\cdot (u\cdot w),\quad\mbox{for all } u,v,w\in V.
\end{align*}

An $\mathcal{O}$-operator on a hom-Lie algebra induces a hom-pre-Lie algebra. In particular, we have the following straightforward proposition.

\begin{proposition}\label{induced hom-pre-Lie algebra}
Let $T:V\rightarrow \mathfrak{g}$ be an $\mathcal{O}$-operator on the hom-Lie algebra $(\mathfrak{g},[~,~],\alpha)$ with respect to the representation $(V,\beta, \rho)$. Then, the $\mathcal{O}$-operator induces a hom-pre-Lie algebra $(V,\cdot_T,\beta )$, where $\cdot_T$ is given by
$$v\cdot_T w= \{Tv,w\}, \quad \mbox{for all  }v, w\in V. $$
\end{proposition}

If $(V,\cdot,\beta)$ is a hom-pre-Lie algebra, then the commutator bracket $[v,w]^c=v\cdot w-w\cdot v$ gives a hom-Lie algebra structure $V^c_{\beta}:=(V,[~,~]^c,\beta)$. It is called the sub-adjacent hom-Lie algebra of the hom-pre-Lie algebra $(V,\cdot,\beta)$.

\begin{proposition}\label{rep associated to O-operator}
Let $T:V\rightarrow \mathfrak{g}$ be an $\mathcal{O}$-operator on the hom-Lie algebra $(\mathfrak{g},[~,~],\alpha)$ with respect to the representation $(V,\beta, \rho)$. Let us define a map $\rho_T:V\rightarrow \mathsf{End}(\mathfrak{g})$ given by
$$\rho_T(v)(x):=[Tv,x]+T\{x,v\},\quad \mbox{for all  }v\in V~~\mbox{and  } x\in\mathfrak{g}.$$
Then, the triplet $(\mathfrak{g},\alpha,\rho_T)$ is a representation of the sub-adjacent hom-Lie algebra $V^c_{\beta}$.
\end{proposition}
\begin{proof}
First, let us show that $\rho_T(\beta(v))(\alpha(x))= \alpha(\rho_T(v)(x))$. The required identity holds by using the facts that 
$$T\circ \beta=\alpha\circ T~~~\mbox{ and }~~~\{\alpha(x),\beta(v)\}=\beta\{x,v\},\quad\mbox{for all } x\in\mathfrak{g},~v\in V.$$ 
In fact,
\begin{align*}
\rho_T(\beta(v))(\alpha(x))&=[T(\beta (v)),\alpha (x)]+T\{\alpha(x),\beta(v)\}\\
&=\alpha([Tv,x]+T\{x,v\})\\
&=\alpha(\rho_T(v)(x)).
\end{align*}
Next, we use the properties of an $\mathcal{O}$-operator  to obtain the following expressions:
\begin{align}\label{rep:eq1}
\rho_T([v,w]^c)(\alpha(x))&=[T\big(\{Tv,w\}-\{Tw,v\}\big),\alpha(x)]+T\{\alpha(x),\{Tv,w\}-\{Tw,v\}\}\\\nonumber
&=[[Tv,Tw],\alpha(x)]+T\{\alpha(x),\{Tv,w\}-\{Tw,v\}\}
\end{align}
\begin{align}\label{rep:eq2}
\rho_T(\beta(v))\rho_T(w)(x)&=\rho_T(\beta(v))([Tw,x]+T\{x,w\})\\\nonumber
&=[T(\beta(v)),[Tw,x]+T\{x,w\}]+T\{[Tw,x]+T\{x,w\},\beta(v)\}\\\nonumber
&=[T(\beta(v)),[Tw,x]] + T\big(\{T(\beta(v)),\{x,w\}\}\big)-T\big(\{T(\{x,w\}),\beta(v)\}\big)\\\nonumber
&\quad+T\{[Tw,x]+T\{x,w\},\beta(v)\}\\\nonumber
&=[\alpha(T(v)),[Tw,x]]+T\big(\{\alpha(T(v)),\{x,w\}\}\big)+T\{[Tw,x],\beta(v)\}
\end{align}
Similarly,
\begin{equation}\label{rep:eq3}
\rho_T(\beta(w))\rho_T(v)(x)=[\alpha(T(w)),[Tv,x]]+T\big(\{\alpha(T(w)),\{x,v\}\}\big) +T\{[Tv,x],\beta(w)\}
\end{equation}

Since the map $\rho:\mathfrak{g}\rightarrow \mathsf{End}(V)$ (denoted by $\{x,v\}:=\rho(x)(v)$) is a representation of the hom-Lie algebra $(\mathfrak{g},[~,~],\alpha)$, it follows that 
\begin{align}
\{\alpha(x),\{Tv,w\}\}-\{\alpha(T(v)),\{x,w\}\}&=\{[x,Tv],\beta(w)\},\\\label{rep:eq5}
\{\alpha(x),\{Tw,v\}\}-\{\alpha(T(w)),\{x,v\}\}&=\{[x,Tw],\beta(v)\}.
\end{align}
Hence, by using equations \eqref{rep:eq1}-\eqref{rep:eq5}, we get
$$\rho_T([v,w]^c)(\alpha(x))= \rho_T(\beta(v))\rho_T(w)(x)-\rho_T(\beta(v))\rho_T(w)(x), ~~\mbox{for all } v,w \in V~~\mbox{and } x\in \mathfrak{g}.$$
Thus, the triplet $(\mathfrak{g},\alpha,\rho_T)$ is a representation of sub-adjacent hom-Lie algebra $V^c_{\beta}$.
\end{proof}

With the above notations, let us consider the cochain complex $(C^*_{\beta,\alpha}(V,\mathfrak{g}),\delta_{\beta,\alpha})$ of the sub-adjacent hom-Lie algebra $V^c_{\beta}$ with coefficients in the representation $(\mathfrak{g},\alpha,\rho_T)$. Recall that the differential $\delta_{\beta,\alpha}$ is given by 
\begin{align}\label{coboundary:subadjacent hom-Lie algebra}
&\delta_{\beta,\alpha} P(v_1,v_2,\ldots,v_{n+1})\\\nonumber
= &\sum_{i=1}^{n+1}(-1)^{i+1}\rho_T(\beta^{n-1}(v_i))\big( P(v_1,v_2,\ldots,\hat{v_i},\ldots,v_{n+1})\big)\\\nonumber
&+\sum_{ i<j }(-1)^{i+j}P([v_i,v_j]^c,\beta(v_1),\ldots,\hat{\beta(v_i)},\ldots,\hat{\beta(v_j)},\ldots,\beta(v_{n+1}))
\end{align}

Thus, there are two different differentials $\delta_\beta$ and $\delta_T:=\{\!\!\{T,- \}\!\!\}$ on the graded vector spaces $C^*_{\beta,\alpha}(V,\mathfrak{g})$. The following proposition shows that both of these differentials yield the same cohomology.

\begin{proposition}
Let $T:V\rightarrow \mathfrak{g}$ be an $\mathcal{O}$-operator on the hom-Lie algebra $(\mathfrak{g},[~,~],\alpha)$ with respect to the representation $(V,\beta, \rho)$. Then, the differentials $\delta_{\beta}$ and $\{\!\!\{T,- \}\!\!\}$ on $C^*_{\beta,\alpha}(V,\mathfrak{g})$ are related by 
$$\delta_T(P)=(-1)^n\{\!\!\{T,P \}\!\!\},\quad \mbox{for any } ~~P\in C^n_{\beta,\alpha}(V,\mathfrak{g})\mbox{ and } n\geq 1.$$
\end{proposition}

\begin{proof} For $P \in C^n_{\beta,\alpha}(V,\mathfrak{g})$ and $n\geq 1$, we have
\begin{align*}
&\{\!\!\{T,P \}\!\!\}(v_1,v_2,\ldots,v_{n+1})\\
=&\sum_{\tau\in S_{n,1}}(-1)^{|\tau|}~ T\Big(\{P(v_{\tau(1)},\ldots,v_{\tau(n)}),\beta^{n-1}(u_{\tau(n+1)})\}\Big)\\
&+(-1)^{n}\Bigg(\sum_{\tau\in S_{1,n}}(-1)^{|\tau|}~\big[\alpha^{n-1}T(v_{\tau(1)}),P(v_{\tau(2)},\ldots,v_{\tau(n+1)})\big]\\
& -\sum_{\tau\in S_{1,1,n-1}}(-1)^{|\tau|}~ P\Big(\{T(v_{\tau(1)}),(u_{\tau(2)})\},\beta(v_{\tau(3)}),\ldots,\beta(v_{\tau(n+1)})\Big)\Bigg)
\end{align*}

\begin{align*}
=&(-1)^n\bigg(\sum_{i=1}^{n+1}(-1)^{i+1}T\{P(v_1,\ldots,\hat{v_i},\ldots,v_{n+1}),\beta^{n-1}(v_i)\}\\
&+\sum_{i=1}^{n+1}(-1)^{i+1}[\alpha^{n-1}\big(T(v_i)\big), P(v_1,\ldots,\hat{v_i},\ldots,v_{n+1})]\\\nonumber
&+\sum_{ i<j }(-1)^{i+j}P\big(\{T(v_i),v_j\}-\{T(v_j),v_i\},\beta(v_1),\ldots,\hat{\beta(v_i)},\ldots,\hat{\beta(v_j)},\ldots,\beta(v_{n+1})\big)\bigg)\\
= &\sum_{i=1}^{n+1}(-1)^{i+1}\rho_T(\beta^{n-1}(v_i))\big( P(v_1,v_2,\ldots,\hat{v_i},\ldots,v_{n+1})\big)\\
&+\sum_{ i<j }(-1)^{i+j}P([v_i,v_j]^c,\beta(v_1),\ldots,\hat{\beta(v_i)},\ldots,\hat{\beta(v_j)},\ldots,\beta(v_{n+1}))\\
=&(-1)^n\delta_T(P)(v_1,v_2,\ldots,v_{n+1}).
\end{align*}
Hence, $\delta_T(P)=(-1)^n\{\!\!\{T,P \}\!\!\},\quad \mbox{for any } ~~P\in C^n_{\beta,\alpha}(V,\mathfrak{g})\mbox{ and } n\geq 1.$
\end{proof}

%\section{Deformations of $\mathcal{O}$-operators on hom-Lie algebras} 
%In this section, we discuss one-parameter deformations of $\mathcal{O}$-operators on a hom-Lie algebra. We interpret both the linear and formal (one-parameter) deformations in terms of the cohomology associated to $\mathcal{O}$-operators. 

%\subsection{Linear deformations}
 %Let $T: V\rightarrow \mathfrak{g}$ be an $\mathcal{O}$-operator on a hom-Lie algebra $(\mathfrak{g},[~,~],\alpha)$ with respect to a representation $(V,\beta,\rho)$. A linear deformation of $T$ is given by a map 
 %$$T_t:=T+t\mathfrak{T}:V\rightarrow \mathfrak{g},\quad \mbox{for some } \mathfrak{T}\in C^1_{\beta,\alpha}(V,\mathfrak{g}),$$
 %such that $T_t $ is an $\mathcal{O}$-operator on a hom-Lie algebra $(\mathfrak{g},[~,~],\alpha)$ with respect to a representation $(V,\beta,\rho)$. 

\section{\normalfont\large\textbf{Deformation of $\mathcal{O}$-operators on regular hom-Lie algebras}}
In this section, we discuss linear and formal one-parameter deformations of $\mathcal{O}$-operators on hom-Lie algebras. In this section, we always assume hom-Lie algebras to be regular and the endomorphism in the representations to be an isomorphism. 

\subsection{Deformation complex of an $\mathcal{O}$-operator on regular hom-Lie algebras}
 We define a deformation complex of $\mathcal{O}$-operators on regular hom-Lie algebras. However, we will see that to get suitable deformation cohomology; we need to define the space of $0$-cochains. For this purpose, we consider a regular hom-Lie algebra $(\mathfrak{g},[~,~],\alpha)$ with a representation $(V,\beta,\rho)$, where $\beta:V\rightarrow V$ is a vector space isomorphism. In this case, an $\mathcal{O}$-operator $T:V\rightarrow \mathfrak{g}$ induces a regular hom-pre-Lie algebra $(V,\cdot_T,\beta)$ and hence, the sub-adjacent hom-Lie algebra $V_{\beta}^c$ is also regular.
Subsequently, from Remark \ref{regular hom-Lie algebra cohomology} we have a modified cochain complex 
$$\bigg(\widetilde{C}^*_{\beta,\alpha}(V,\mathfrak{g}):=\bigoplus_{n\geq 0}~C^n_{\beta,\alpha}(V,\mathfrak{g}),~\delta_{\beta,\alpha}\bigg),$$
where the space of $0$-cochains are given by 
$$C^0_{\beta,\alpha}(V,\mathfrak{g}):=\{x\in \mathfrak{g}| \alpha(x)=x\},$$
and the differential $\delta_{\beta,\alpha}:C^0_{\alpha,\beta}(V,\mathfrak{g})\rightarrow C^1_{\alpha,\beta}(V,\mathfrak{g})$ is defined by 
\begin{equation}\label{diff_0:1}
\delta_{\beta,\alpha}(x)(v):=\rho_T(\beta^{-1}(v))(x)\quad\mbox{for }v\in V,~x\in C^0_{\beta,\alpha}(V,\mathfrak{g}).
\end{equation}

Note that in this case we can extend the bracket $\{\!\!\{-,- \}\!\!\}$ defined by equation \eqref{derived bracket} to $\widetilde{C}^*_{\beta,\alpha}(V,\mathfrak{g})$. In particular, for $x,y\in\mathfrak{g},$ and $P\in C^n_{\beta,\alpha}(V,\mathfrak{g})$, the bracket $\{\!\!\{-,- \}\!\!\}$ is given by 

\begin{align}
\nonumber
\{\!\!\{x,y\}\!\!\}=&[x,y],\\\nonumber
\{\!\!\{P,x\}\!\!\}(v_1,v_2,\ldots,v_{n})=&\sum_{\tau\in S_{1,n-1}}(-1)^{|\tau|}~ P\Big(\{x,\beta^{-1}(v_{\tau(1)})\},v_{\tau(2)},\ldots, v_{\tau(n)} \Big)\\\nonumber
&+[\alpha^{-1}P(v_{1},\ldots,v_{n}),\alpha^{n-1}(x)].
\end{align}
In particular, 
\begin{equation}\label{diff_0:2}
\{\!\!\{T,x\}\!\!\}(v)= T(\{x,\beta^{-1}(v)\})+[\alpha^{-1}T(v),(x)].
\end{equation}

By equations \eqref{diff_0:1} and \eqref{diff_0:2}, it is clear that $\delta_{\beta,\alpha}$ coincides with $\delta_T$ at $0$-degree elements in $\widetilde{C}^*_{\beta,\alpha}(V,\mathfrak{g})$. i.e.,  
\begin{align*}
\delta_{\beta,\alpha}(x)(v)=&\rho_T(\beta^{-1}(v))(x)\\
=&[T(\beta^{-1}(v)),x]+T\{x,\beta^{-1}(v)\}\\
=& [\alpha^{-1}T(v),x] + T(\{x,\beta^{-1}(v)\})\\
=&\{\!\!\{T,x\}\!\!\}(v),
\end{align*}
for all $v\in V,~x\in C^0_{\beta,\alpha}(V,\mathfrak{g})$. Therefore, the cohomologies of the complexes $(\widetilde{C}^*_{\beta,\alpha}(V,\mathfrak{g}),\delta_{\beta,\alpha})$ and $(\widetilde{C}^*_{\beta,\alpha}(V,\mathfrak{g}),\delta_T)$ are the same and denoted by $\widetilde{H}^*_{\beta,\alpha}(V,\mathfrak{g})$. In the sequel, we show that the cohomology $\widetilde{H}^*_{\beta,\alpha}(V,\mathfrak{g})$ is deformation cohomology for an $\mathcal{O}$-operator on the regular hom-Lie algebra $(\mathfrak{g},[~,~],\alpha)$ with respect to a representation $(V,\beta,\rho)$.

\subsection{Linear deformations}
Let $T:V\rightarrow \mathfrak{g}$ be an $\mathcal{O}$-operator on a hom-Lie algebra $(\mathfrak{g},[~,~],\alpha)$  with respect to a representation $(V,\beta,\rho)$. Let us consider a linear sum $T_t:=T+t\mathfrak{T}$ for some element $\mathfrak{T}\in C^1_{\beta,\alpha}(V,\mathfrak{g})$. If $T_t$ is an $\mathcal{O}$-operator on the hom-Lie algebra $(\mathfrak{g},[~,~],\alpha)$ with respect to the representation $(V,\beta,\rho)$, then $T_t$ is called a linear deformation of $T$ generated by the element $\mathfrak{T}$. The map $T_t=T+t\mathfrak{T}$ is a linear deformation of $T$ if it satisfies the following identities
\begin{align*}
T_t\circ \beta&=\alpha\circ T_t,\\
[T_t(v),T_t(w)]&=T_t\big(\{T_t(v),w\}-\{T_t(w),v\}\big),\quad\mbox{for all }v,w\in V.
\end{align*}
Equivalently, 
\begin{equation}\label{commuting condition}
\mathfrak{T}\circ \beta =\alpha\circ \mathfrak{T}
\end{equation}
and for all $v,w \in V$, we get
\begin{equation}\label{cocycle condition for linear def}
[T(v),\mathfrak{T}(w)]+[\mathfrak{T}(v),T(w)]=T\big(\{\mathfrak{T}(v),w\}-\{\mathfrak{T}(w),v\}\big)+\mathfrak{T}\big(\{T(v),w\}-\{T(w),v\}\big),
\end{equation} 
\begin{equation}\label{O-operator condition for generator}
[\mathfrak{T}(v),\mathfrak{T}(w)]=\mathfrak{T}\big(\{\mathfrak{T}(v),w\}-\{\mathfrak{T}(w),v\}\big).
\end{equation}
Thus, $T_t$ is a linear deformation of $T$ if and only if conditions \eqref{commuting condition}-\eqref{O-operator condition for generator} hold. Observe that condition \eqref{cocycle condition for linear def} implies that $\delta_T(\mathfrak{T})=0$. Moreover, it follows from \eqref{commuting condition} and \eqref{O-operator condition for generator} that the map $\mathfrak{T}$ is an $\mathcal{O}$-operator on the hom-Lie algebra $(\mathfrak{g},[~,~],\alpha)$ with respect to the representation $(V,\beta,\rho)$.

\begin{definition}
Two linear deformations $T_t^1:=T+t\mathfrak{T}_1$ and $T_t^2:=T+t\mathfrak{T}_2$ are said to be equivalent if there exist an element $x\in \mathfrak{g}$ such that $\alpha(x)=x$ and the pair $(\mathsf{Id}_{\mathfrak{g}}+t\mathsf{ad}^\dagger_x,\mathsf{Id}_V+t\rho(x)^\dagger)$ is a homomorphism of $\mathcal{O}$-operators from $T_t^1$ to $T_t^2$. 
\end{definition}
Let us recall from Definition \ref{morphism of O-operators} that the pair $(\mathsf{Id}_{\mathfrak{g}}+t\mathsf{ad}^\dagger_x,\mathsf{Id}_V+t\rho(x)^\dagger)$ is a homomorphism of $\mathcal{O}$-operators from $T_t^1$ to $T_t^2$ if the following conditions are satisfied
\begin{enumerate}[label=(\roman*)]
\item The map $\mathsf{Id}_{\mathfrak{g}}+t\mathsf{ad}^\dagger_x$ is a hom-Lie algebra homomorphism,
\item $\beta\circ(\mathsf{Id}_{\mathfrak{g}}+t\rho(x)^\dagger)=(\mathsf{Id}_{\mathfrak{g}}+t\rho(x)^\dagger)\circ\beta,$ 
\item $(T+t\mathfrak{T}_2)\circ(\mathsf{Id}_V+t\rho(x)^\dagger)=(\mathsf{Id}_{\mathfrak{g}}+t\mathsf{ad}_x^\dagger)\circ(T+t\mathfrak{T}_1),$ 
\item $\rho\big((\mathsf{Id}_{\mathfrak{g}}+t\mathsf{ad}_x^\dagger)(y)\big)\big((\mathsf{Id}_{\mathfrak{g}}+t\rho(x)^\dagger)(v)\big)=(\mathsf{Id}_{\mathfrak{g}}+t\rho(x)^\dagger)\big(\rho(y)(v)\big),$ for all $y\in \mathfrak{g}$ and $v\in V$.
\end{enumerate}
From the condition (i), 
$$(\mathsf{Id}_{\mathfrak{g}}+t\mathsf{ad}_x^\dagger)[y,z]=[(\mathsf{Id}_{\mathfrak{g}}+t\mathsf{ad}_x^\dagger)(y),(\mathsf{Id}_{\mathfrak{g}}+t\mathsf{ad}_x^\dagger)(z)],\quad\mbox{for all }y,z\in \mathfrak{g}.$$
On comparing the coefficients of $t^2$ on both sides of the above identity, we get 
$$[[x,\alpha^{-1}(y)],[x,\alpha^{-1}(z)]]=0,\quad\mbox{for all }y,z\in\mathfrak{g}.$$ 
Clearly, invertibility of $\alpha$ implies that the element $x$ satisfies
\begin{equation}\label{equ-linear:cond1}
[[x,y],[x,z]]=0,\quad\mbox{for all }y,z\in\mathfrak{g}.
\end{equation}
Since $\alpha(x)=x$, the condition (ii) holds. It easily follows that the condition (iii) is equivalent to the following identities
\begin{equation}\label{equ-linear:cond2}
\mathfrak{T}_1(v)-\mathfrak{T}_2(v)=T\rho(x)(\beta^{-1}(v))+[T\beta^{-1}(v),x]=\delta_T(x)(v),
\end{equation}
\begin{equation}\label{equ-linear:cond3}
\mathfrak{T}_2\rho(x)(v)=[x,\mathfrak{T}_1(v)],\quad\mbox{for all }v\in V.
\end{equation}
The last condition (iv) implies that the element $x$ also satisfies
$$\rho([x,\alpha^{-1}(y)])\rho(x)(\beta^{-1}(v))=0, \quad\mbox{for all } y\in \mathfrak{g}, ~v\in V,$$
equivalently, by using invertibility of $\alpha$ and $\beta$, we have
\begin{equation}\label{equ-linear:cond4}
\rho([x,y])\rho(x)(v)=0, \quad\mbox{for all } y\in \mathfrak{g}, ~v\in V.
\end{equation}
\begin{theorem}
Let $T:V\rightarrow \mathfrak{g}$ be an $\mathcal{O}$-operator. Let $T_t^1:=T+t\mathfrak{T}_1$ and $T_t^1:=T+t\mathfrak{T}_1$ be two equivalent linear deformations of $T$. Then, $\mathfrak{T}_1$ and $\mathfrak{T}_2$ belong to the same cohomology class in $H^1_{\beta,\alpha}(V,\mathfrak{g})$.
\end{theorem}
\begin{proof}
The proof follows from equation \eqref{equ-linear:cond2}.
\end{proof}
 
\begin{definition}
Let $T: V\rightarrow \mathfrak{g}$ be an $\mathcal{O}$-operator on a hom-Lie algebra $(\mathfrak{g},[~,~],\alpha)$ with respect to a representation $(V,\beta,\rho)$. A linear deformation $T_t:T+t\mathfrak{T}$ is said to be trivial if it is equivalent to the deformation $T_0=T$.
\end{definition}

\begin{definition}
 Let $T: V\rightarrow \mathfrak{g}$ be an $\mathcal{O}$-operator on a hom-Lie algebra $(\mathfrak{g},[~,~],\alpha)$ with respect to a representation $(V,\beta,\rho)$. An element $x\in \mathfrak{g}$ is called a Nijenhuis element associated to the operator $T$ if $x$ satisfies $\alpha(x)=x$, the identities \eqref{equ-linear:cond1}, \eqref{equ-linear:cond4}, and the identity
$$[x,T\rho(x)(v)+[T(v),x]]=0,\quad \mbox{for all } v\in V.$$ 
\end{definition}
Let us denote the set of Nijenhuis elements associated to the $\mathcal{O}$-operator $T$ by $\mathsf{Nij}(T)$.

\begin{theorem}\label{trivial deformations}
Let $T: V\rightarrow \mathfrak{g}$ be an $\mathcal{O}$-operator on a hom-Lie algebra $(\mathfrak{g},[~,~],\alpha)$ with respect to a representation $(V,\beta,\rho)$. For any element $x\in \mathsf{Nij}(T),$ the linear deformation $T_t:T+t\mathfrak{T}$ generated by $\mathfrak{T}:=\delta_T(x)$ is a trivial deformation of $T$.
\end{theorem}
\begin{proof}
First, we need to show that $T_t:T+t\mathfrak{T}$ is a linear deformation generated by $\mathfrak{T}:=\delta_{\beta,\alpha}(x)$, where $x\in\mathsf{Nij}(T)$. For this purpose, we need to show that $\mathfrak{T}$ satisfies the equations \eqref{commuting condition}, \eqref{cocycle condition for linear def}, and \eqref{O-operator condition for generator}. By definition of $\delta_{\beta,\alpha}$ at $0$-cochains it is clear that 
$$\alpha\circ \mathfrak{T}(v)=\alpha\circ \delta_{\beta,\alpha}(x)(v)=\delta_{\beta,\alpha}(x)(\beta(v))=\mathfrak{T}(\beta(v)), \quad\mbox{for all }v\in V.$$
i.e., $\mathfrak{T}$ satisfies the equation \eqref{commuting condition}. Since $\mathfrak{T}=\delta_{\beta,\alpha}(x)$, the equation \eqref{cocycle condition for linear def} holds trivially. Moreover, by using the identity $\alpha(x)=x$ and a straightforward calculation similar to the Lie algebra case in \cite{Sheng3}, it follows that $\mathfrak{T}$ satisfies the equation \eqref{O-operator condition for generator}. 

Now, we need to show that the linear deformation $T_t$ is trivial. Since $x\in \mathsf{Nij}(T)$, it immediately follows that the pair $(\mathsf{Id}_\mathfrak{g}+t\mathsf{ad}_x^{\dagger},\mathsf{Id}_V+t\rho(x)^\dagger)$ is a homomorphism of $\mathcal{O}$-operators from $T_t$ to $T$. 
\end{proof}
\subsection{Formal deformations}
Let $(\mathfrak{g},[~,~],\alpha)$ be a hom-Lie algebra  with a representation $(V,\beta,\rho)$. Let $k[[t]]$ be the formal power series ring in one variable $t$ and $\mathfrak{g}[[t]]$ be the formal power series in $t$ with coefficients in $\mathfrak{g}$. Then the triplet $(\mathfrak{g}[[t]],[~,~]_t,\alpha_t)$ is a hom-Lie algebra, where the bracket $[~,~]_t$ and the structure map $\alpha_t$ are obtained by extending $[~,~]$ and the map $\alpha$ linearly over the ring $k[[t]]$. Moreover, the map $\rho:\mathfrak{g}\rightarrow \mathsf{End}(V)$ and $\beta:V\rightarrow V$ can be extended linearly over $k[[t]]$ to obtain $k[[t]]$-linear maps $\rho_t:\mathfrak{g}[[t]]\otimes V[[t]] \rightarrow V[[t]]$ and $\beta_t:V[[t]]\rightarrow V[[t]]$. Then the triplet $(V[[t]],\beta_t,\rho_t)$ is a hom-Lie algebra representation of $(\mathfrak{g}[[t]],[~,~]_t,\alpha_t)$.      

\begin{definition}
 Let $T: V\rightarrow \mathfrak{g}$ be an $\mathcal{O}$-operator on a hom-Lie algebra $(\mathfrak{g},[~,~],\alpha)$ with respect to a representation $(V,\beta,\rho)$. A formal deformation of $T$ is given by 
$$\textstyle{T_t=T_0+\sum\limits_{i\geq1}t^i T_i }, \quad\mbox{ with }  T_0=T,~T_i\in C^1_{\beta,\alpha}(V,\mathfrak{g})$$ 
such that $T_t:V[[t]]\rightarrow\mathfrak{g}[[t]]$ is an $\mathcal{O}$-operator on $(\mathfrak{g}[[t]],[~,~]_t,\alpha_t)$ with respect to the representation $(V[[t]],\beta_t,\rho_t)$. 
\end{definition}
Equivalently, 
\begin{equation}\label{fdef:eq1}
T_t(\beta(v))=\alpha (T_t(v)),
\end{equation}
\begin{equation}\label{fdef:eq2}
[T_t v,T_t w]=T_t\big(\{T_t v, w\}-\{T_t w,v\}\big),\quad\mbox{ for all }  u,v\in V.
\end{equation}
Note that condition \eqref{fdef:eq1} holds trivially since $T_i\in C^1_{\beta,\alpha}(V,\mathfrak{g}),$ for all $i\geq 0$. For $k\geq 0$, if we compare the coefficients of $t^k$ from both sides of equation \eqref{fdef:eq2}, then we obtain the following system of equations
\begin{equation}\label{fdef:system}
\sum\limits_{i+j=k}[T_i v,T_j w]=\sum\limits_{i+j=k}T_i\big(\{T_j v, w\}-\{T_j w,v\}\big),\quad\mbox{ for } k=0,1,2,\ldots.
\end{equation}

The 1-cochain $T_1\in C^1_{\beta,\alpha}(V,\mathfrak{g})$ is called the infinitesimal of the deformation $T_t$. More generally, if $T_i=0$ for $1\leq i\leq n-1$ and $T_n$ is a non-zero cochain, then $T_n$ is called the $n$-infinitesimal of the deformation $T_t$. From equation \eqref{fdef:system}, the case $k=1$ yields the expression
$$[T_1 v,T w]+[T v,T_1 w]=T_1\big(\{T v, w\}-\{T w,v\}\big)+T\big(\{T_1 v, w\}-\{T_1 w,v\}\big),\quad\mbox{ for all }  v,w\in V,$$
which is equivalent to the condition: $\delta_T(T_1)=0$. Therefore, we get the following proposition.

\begin{proposition}\label{infinitesimal}
The infinitesimal of the deformation $T_t$ is a $1$-cocycle in the cohomology of the $\mathcal{O}$-operator $T$. More generally, the $m$-infinitesimal is a $1$-cocycle. 
\end{proposition}

Now, we consider the equivalence of two formal deformations of an $\mathcal{O}$-operator. The definition is motivated from the Lie algebra case \cite{Sheng3}.
\begin{definition}\label{Def:equivalence}
Two deformations $T_t$ and $\overline{T}_t$ of the $\mathcal{O}$-operator $T$ are said to be equivalent if there exists an element $x\in C^{0}_{\beta,\alpha}(V,\mathfrak{g})$, $k$-linear maps $\phi^{\mathfrak{g}}_i:\mathfrak{g}\rightarrow\mathfrak{g}$ and $\phi^V_i:V\rightarrow V$, for $i\geq 2$, such that the pair $(\phi^{\mathfrak{g}}_t,\phi^V_t),$ consisting of
$$\phi^{\mathfrak{g}}_t=\mathsf{Id}_{\mathfrak{g}}+t(\mathsf{ad}^{\dagger}_x)+\sum_{i\geq 2}t^i \phi^{\mathfrak{g}}_i \quad \mbox{and}\quad \phi^V_t=\mathsf{Id}_{V}+t\rho(x)^{\dagger}+\sum_{i\geq 2}t^i \phi^V_i,$$
is a formal isomorphism from $T_t$ to ${T}^{\prime}_t$. Here, for $x\in C^{0}_{\beta,\alpha}(V,\mathfrak{g})$, the maps $\mathsf{ad}_x^{\dagger}:\mathfrak{g}\rightarrow \mathfrak{g}$ and $\rho(x)^{\dagger}:V\rightarrow V$ are given by 
$$\mathsf{ad}^{\dagger}_x(y):=\alpha^{-1}\big(\mathsf{ad}_x(y)\big)\quad\mbox{and  }\rho(x)^{\dagger}(v):=\beta^{-1}\big(\rho(x)(v)\big),\quad\mbox{for }y\in \mathfrak{g}, ~v\in V.$$  
\end{definition}

With the above notations, for all $y,z\in \mathfrak{g}$ and $v\in V$, the equivalence of two deformations $T_t$ and $\overline{T}_t$ gives the following conditions
\begin{enumerate}
\item $\phi^{\mathfrak{g}}_t\circ T_t=\overline{T}_t\circ \phi^{V}_t,$\\\vspace{-4.5mm}
\item $\phi^{\mathfrak{g}}_t[y,z]=[\phi^{\mathfrak{g}}_t(y),\phi^{\mathfrak{g}}_t(z)]$,\\\vspace{-4mm}
\item $\rho_t(\phi^{\mathfrak{g}}_t(y))(\phi^V_t(v))=\phi^V_t\big(\rho(y)(v)\big)$,\\\vspace{-4.5mm}
\item $\phi^{\mathfrak{g}}_t\circ \alpha=\alpha\circ\phi^{\mathfrak{g}}_t~~$ and $~~\phi^V_t\circ\beta=\beta\circ\phi^V_t$.
\end{enumerate}

On comparing the coefficients of $t$ from both sides of the condition $(1)$, we get 
\begin{align*}
T_1(v)-\overline{T}_1(v)=&T\beta^{-1}(\{x,v\})-\alpha^{-1}[x,Tv]\\
=&T(\{x,\beta^{-1}(v)\})+[T(\beta^{-1}(v)),x]\\
=&\rho_T(\beta^{-1}(v))(x)=(\delta_{\beta,\alpha}(x))(v)            
\end{align*}
Consequently, we obtain the following result.
\begin{proposition}
The infinitesimals of equivalent deformations belong to the same cohomology class in $\widetilde{H}^2_{\beta,\alpha}(V,\mathfrak{g})$.
\end{proposition}

\subsection{Obstructions in extending a finite order deformation to the next order}

Let a linear map $T:V\rightarrow \mathfrak{g}$ be an $\mathcal{O}$-operator on a hom-Lie algebra $(\mathfrak{g},[~,~],\alpha)$ with respect to a representation $(V,\beta,\rho)$. An order $n$ deformation of the $\mathcal{O}$-operator $T$ is given by a $k[[t]]/(t^{n+1})$-linear map 
$${T_t=T_0+\sum\limits^n_{i=1}~t^i ~T_i }, \quad\mbox{ with  }  T_0=T,~~T_i\in C^1_{\beta,\alpha}(V,\mathfrak{g})$$ 
such that 
$$[T_t(v),T_t(w)]=T_t(\{T_t(v),w\}-\{T_t(w),v\}) \quad\mbox{modulo}~~ t^{n+1}, \quad \mbox{for all }v,w\in V.$$
Equivalently,
$$\sum\limits_{\substack{i+j=k\\ ~i,j\geq 0}}\{\!\!\{T_i,T_j\}\!\!\}=0\quad\quad\mbox{for any }~~k=0,1,\ldots,n.$$

\begin{definition}
Let $\textstyle{T_t=\varphi_0+\sum^n_{i=1}t^i T_i }$ be an order $n$ deformation of the $\mathcal{O}$-operator $T$. We say that $T_t$ extends to a deformation of order $n+1$ if there exists a $1$-cochain $T_{n+1}\in C^1_{\beta,\alpha}(V,\mathfrak{g})$ such that $\widetilde{T}_t=T_t+t^{n+1} T_{n+1} $ is a deformation of order $n+1$.
\end{definition}

Let us observe that the map $\widetilde{T}_t:=T_t+t^{n+1}T_{n+1}$ is an extension of the order $n$ deformation $T_t$ if and only if 
$$\sum_{\substack{i+j=n+1\\i,j\geq 0}}\{\!\!\{T_i,T_j\}\!\!\}=0.$$

\begin{definition}
Let $T_t$ be an order $n$ deformation of the $\mathcal{O}$-operator $T$. Let us consider a $2$-cochain $\Theta_T \in C^2_{\beta,\alpha}(V,\mathfrak{g})$ defined as follows
\begin{equation}\label{Obst}
\Theta_T =-1/2\sum_{i+j=n+1;~i,j>0}\{\!\!\{T_i,T_j\}\!\!\}.
\end{equation}
The $2$-cochain $\Theta_T$ is called the obstruction cochain for extending the deformation $T_t$ of order $n$ to a deformation of order $n+1$. 
From equation \eqref{Obst} and using graded Jacobi identity of the bracket $\{\!\!\{-,-\}\!\!\}$, it follows that $\Theta_T$ is a $2$-cocycle.
\end{definition}

\begin{theorem}\label{hom-Obst}
Let $T_t$ be an order $n$ deformation of $T$. Then the deformation $T_t$ extends to a deformation of order $n+1$ if and only if the cohomology class of the $2$-cocycle $\Theta_T$ vanishes.

\begin{proof}
Let us assume that the order $n$ deformation $\textstyle{T_t=\varphi_0+\sum^n_{i=1}t^i T_i }$ extends to a deformation of order $n+1$. Then, there exists an element $T_{n+1}\in C^1_{\beta,\alpha}(V,\mathfrak{g})$ such that 
$\widetilde{T_t}=T_t+t^{n+1} T_{n+1}$ is an extension of the deformation $T_t$. Thus,  
$$\sum_{\substack{i+j=n+1\\ i,j\geq 0}}\{\!\!\{T_i,T_j\}\!\!\}=0,$$
i.e.,
$$\{\!\!\{T,T_{n+1}\}\!\!\}=-\frac{1}{2} \sum_{\substack{i+j=n+1\\ i,j> 0}}\{\!\!\{T_i,T_j\}\!\!\}.$$
It follows that $\Theta_T=\delta_T(T_{n+1})=\delta_{\beta,\alpha}(T_{n+1})$. Hence, the cohomology class of $\Theta_T$ vanishes.

Conversely, let us assume that $\Theta_T$ is a coboundary. So, there exists a $1$-cochain $T_{n+1}$ such that 
$$
\Theta_T==\delta_{\beta,\alpha}(T_{n+1}). 
$$
Define a map $\widetilde{T_t}:V\rightarrow \mathfrak{g}$ as follows
$$
\widetilde{T_t}=T_t+t^{n+1}T_{n+1}.
$$
Then,
$$\Theta_T=-1/2\sum_{\substack{i+j=n+1\\i,j>0}}\{\!\!\{T_i,T_j\}\!\!\}=\delta_T(T_{n+1}).$$
Since, $\delta_{\beta,\alpha}(T_{n+1})=\delta_T(T_{n+1})=\{\!\!\{T,T_{n+1}\}\!\!\}$, we obtain the following expression
$$\sum_{\substack{i+j=n+1\\i,j\geq 0}}\{\!\!\{T_i,T_j\}\!\!\}=0.$$ 
Therefore, the deformation $T_t$ of order $n$ extends to the deformation $\widetilde{T_t}$ of order $n+1$.
\end{proof}

\end{theorem}

\begin{corollary}
Let $T:V\rightarrow \mathfrak{g}$ be an $\mathcal{O}$-operator on a hom-Lie algebra $(\mathfrak{g},[~,~],\alpha)$ with respect to a representation $(V,\beta,\rho)$. If $H^2_{\beta,\alpha}(V,\mathfrak{g})=0$, then any $1$-cocycle in $C^1_{\beta,\alpha}(V,\mathfrak{g})$ is an infinitesimal of some formal deformation of the $\mathcal{O}$-operator $T$.
\end{corollary}

\section{Applications}
In this section, we describe deformations of $s$-Rota-Baxter operators (of weight $0$) and skew-symmetric $r$-matrices on hom-Lie algebras as particular cases of $\mathcal{O}$-operators on hom-Lie algebras. 
\subsection{Rota-Baxter Operators on hom-Lie algebras}
In this subsection, we consider Rota Baxter operators of weight $0$. Let us recall the definition of $s$-Rota-Baxter operator on a hom-Lie algebra from definition \ref{def:Rota-Baxter operators}. In particular, for any non-negative integer $s$, a linear operator $\mathcal{R}: \mathfrak{g} \rightarrow \mathfrak{g}$ is called an $s$-Rota-Baxter operator of weight $0$ on a hom-Lie algebra $(\mathfrak{g},[~,~],\alpha)$ if $\mathcal{R}\circ \alpha= \alpha\circ \mathcal{R}$ and the following identity is satisfied
\begin{equation*} 
\quad\quad\quad\quad\quad[\mathcal{R}(x), \mathcal{R}(y)]= \mathcal{R}([\alpha^s \mathcal{R}(x), y]+ [x, \alpha^s \mathcal{R}(y)],\quad\mbox{for all }x,y\in \mathfrak{g}.
\end{equation*}

\begin{proposition}\label{induced hom-Lie algebra}
Let $\mathcal{R}$ be an $s$-Rota-Baxter operator on a hom-Lie algebra $(\mathfrak{g}, [~,~], \alpha)$. Then $\mathcal{R}$ induces a hom-Lie algebra structure $(\mathfrak{g},[~,~]_\mathcal{R},\alpha)$, where the bracket $[~,~]_\mathcal{R}$ is given by 
$$[x, y]_\mathcal{R}= [\alpha^s \mathcal{R}x, y] + [x, \alpha^s \mathcal{R}y].$$ 
\end{proposition}
\begin{proof}
From Remark \ref{5.2} and Proposition \ref{induced hom-pre-Lie algebra}, one obtains the induced hom-pre-Lie algebra structure. Then the hom-Lie algebra $(\mathfrak{g},[~,~]_\mathcal{R},\alpha)$ is the sub-adjacent hom-Lie algebra to this induced hom-pre-Lie algebra.
\end{proof}

From remark \ref{5.2}, any $s$-Rota-Baxter operator on a hom-Lie algebra $(\mathfrak{g},[~,~],\alpha)$ is simply an $\mathcal{O}$-operator on $(\mathfrak{g}, [~,~], \alpha)$ with respect to the $\alpha^s$-adjoint representation. Consequently, the results developed in Section $3$ also hold for $s$-Rota-Baxter operators on hom-Lie algebras. More precisely, the $\alpha^s$-adjoint representation of the hom-Lie algebra $(\mathfrak{g},[~,~],\alpha)$ on itself induces a graded Lie algebra structure $\big(C^*_{\alpha}(\mathfrak{g},\mathfrak{g}),\{\!\!\{-,-\}\!\!\}\big)$, as described in Subsection 3.1, and we have the following result.
\begin{theorem}
A linear map $\mathcal{R}: \mathfrak{g} \rightarrow \mathfrak{g}$ is an $s$-Rota-Baxter operator on a hom-Lie algebra $(\mathfrak{g},[~,~],\alpha)$ if and only if $\mathcal{R}$ is a Maurer-Cartan element of the graded Lie algebra $\big(C^*_{\alpha}(\mathfrak{g},\mathfrak{g}),\{\!\!\{-,-\}\!\!\}\big)$. Thus, an $s$-Rota-Baxter operator $\mathcal{R}$ on $(\mathfrak{g},[~,~],\alpha)$ induces a differential graded Lie algebra structure $\big(C^*_{\alpha}(\mathfrak{g},\mathfrak{g}),\{\!\!\{~,~\}\!\!\},\delta_\mathcal{R}:=\{\!\!\{\mathcal{R},~\}\!\!\}\big)$.  
\end{theorem}

Let us observe that for a linear map $\mathcal{R}^\prime:\mathfrak{g}\rightarrow \mathfrak{g}$, the sum $\mathcal{R}+\mathcal{R}^\prime$ is an $s$-Rota-Baxter operator on $(\mathfrak{g},[~,~],\alpha)$ if and only if 
$$\delta_\mathcal{R}(\mathcal{R}^{\prime})+\frac{1}{2}\{\!\!\{\mathcal{R}^\prime,\mathcal{R}^\prime\}\!\!\}=0,$$
i.e., $\mathcal{R}^\prime$ is a Maurer-Cartan element of the differential graded Lie algebra $\big(C^*_{\alpha}(\mathfrak{g},\mathfrak{g}),\{\!\!\{~,~\}\!\!\},\delta_\mathcal{R}\big)$.

If the hom-Lie algebra $(\mathfrak{g},[~,~],\alpha)$ is regular, then one may describe linear and formal deformations of $s$-Rota-Baxter operators on $(\mathfrak{g},[~,~],\alpha)$, for any integer $s$. This description follows from the deformation theory for $\mathcal{O}$-operators in Section $4$. In particular, let us write down some definitions and results on linear deformations of $s$-Rota-Baxter operators.

\begin{definition} Let $\mathcal{R}$ be an $s$-Rota–Baxter operator on a regular hom-Lie algebra $(\mathfrak{g}, [~,~], \alpha)$.\\
(i) Let $R: \mathfrak{g}\longrightarrow \mathfrak{g}$ be a linear operator. If a $t$-parametrized family $\mathcal{R}_t := \mathcal{R} + t R$ is an $s$-Rota–Baxter operator on $(\mathfrak{g}, [~,~], \alpha),$ for all $t \in K$, then we say that $R$ generates a linear deformation of $\mathcal{R}$.\\
(ii) Let $\mathcal{R}^1_t= \mathcal{R} + t R_1$ and $\mathcal{R}^2_t:= \mathcal{R} + t R_2$ be two linear deformations of $\mathcal{R}$ generated
by $R_1$ and $R_2$ respectively. They are said to be equivalent if there exists an $x \in \mathfrak{g}$
such that $\alpha(x)=x$ and the pair $(\mathsf{Id}_\mathfrak{g} + t \mathsf{ad}^\dagger_x , \mathsf{Id}_\mathfrak{g} + t\rho_s(x)^\dagger)$ is a homomorphism from $\mathcal{R}^2_t$ to $\mathcal{R}^1_t$, where $$\mathsf{ad}^{\dagger}_x(y):=\alpha^{-1}[x, y]~~ \mbox{and  }\rho_s(x)^{\dagger}(y):=\alpha^{-1}[\alpha^s x, y]~~\mbox{for } x, y\in \mathfrak{g}.$$ 
%(iii) A linear deformation $\mathcal{R}_t = \mathcal{R} + tR$ of $\mathcal{R}$ is said to be trivial if there exists an $x \in \mathfrak{g}$ such that $(\mathsf{Id}_\mathfrak{g} + t \mathsf{ad}^\dagger_x , \mathsf{Id}_\mathfrak{g} + t\rho_s(x)^\dagger)$ is a homomorphism from $\mathcal{R}_t$ to $\mathcal{R}$, where $ad^\dagger_x$ and $\rho_s(x)^\dagger$ are defined as above.
\end{definition}

Let us consider the hom-Lie algebra $(\mathfrak{g},[-,-]_\mathcal{R},\alpha)$ given by proposition \ref{induced hom-Lie algebra}. From Proposition \ref{rep associated to O-operator}, we obtain a representation $\rho_\mathcal{R}:\mathfrak{g} \rightarrow \mathsf{End}(\mathfrak{g})$ given by
$$\rho_\mathcal{R}(x)(y):=[\mathcal{R}x,y]+\mathcal{R}[\alpha^s(y),x],\quad \mbox{for all  }x,y \in\mathfrak{g}.$$
The triplet $(\mathfrak{g},\alpha,\rho_\mathcal{R})$ is a representation of the hom-Lie algebra $(\mathfrak{g},[~,~]_\mathcal{R},\alpha)$. Then, from Subsection $3.3$, we obtain the extended cochain complex $\big(\widetilde{C}^*_{\alpha,\alpha}(\mathfrak{g},\mathfrak{g}),\delta_{\alpha,\alpha}\big)$ (with $0$-cochains) for the hom-Lie algebra $(\mathfrak{g},[~,~]_\mathcal{R},\alpha)$ with coefficients in the representation $(\mathfrak{g},\alpha,\rho_\mathcal{R})$. This complex serves as deformation complex for an $s$-Rota-Baxter operator on a regular hom-Lie algebra $(\mathfrak{g},[~,~],\alpha)$. Moreover, the following proposition holds.

\begin{proposition} Let $\mathcal{R}$ be an $s$-Rota–Baxter operator on a hom-Lie algebra $\mathfrak{g}$. If
$R$ generates a linear deformation of $\mathcal{R}$, then $R$ is a $1$-cocycle. Moreover, if two linear
deformations of $\mathcal{R}$ generated by $R_1$ and $R_2$ are equivalent, then $R_1$ and $R_2$ determine the
same cohomology class.
\end{proposition}

\begin{definition}
Let $(\mathfrak{g},[~,~],\alpha)$ be a regular hom-Lie algebra and $\mathcal{R}$ be an $s$-Rota-Baxter operator on it. 
Then $x\in \mathfrak{g}$ is called a Nijenhuis
element associated to $s$-Rota-Baxter operator $\mathcal{R}$ if
\begin{align*}
\alpha(x)&=x,\\
[x, [Ry, x] + R[x, y]] &= 0, \quad\mbox{for all  } y \in \mathfrak{g} \\
[[x, y], [x, z]]& = 0,\quad\mbox{for all  }  y, z \in \mathfrak{g}.
\end{align*}
Let us denote the set of all Nijenhuis elements associated to $\mathcal{R}$ by $\mathsf{Nij}(\mathcal{R})$.
\end{definition}

Thus, it is easy to see that the Theorem \ref{trivial deformations} leads to the following result.
\begin{proposition}
 Let $\mathcal{R}$ be an $s$-Rota–Baxter operator on a regular hom-Lie algebra $(\mathfrak{g},[~,~],\alpha)$. If
$R$ generates a trivial linear deformation of $\mathcal{R}$, then it induces a Nijenhuis element. Conversely, for any Nijenhuis element $x\in \mathsf{Nij}(\mathcal{R})$, the linear sum $\mathcal{R}_t = \mathcal{R} + t R$ with $R= \delta_{\alpha}^s(x)$ is a
trivial linear deformation of $\mathcal{R}$. 
\end{proposition}

Note that one can also deduce results on formal deformations of $s$-Rota-Baxter operators (of weight $0$) on regular hom-Lie algebras following the results in Subsection $4.3$.

\subsection{Skew-symmetric r-matrices on regular hom-Lie algebras}
Let $(\mathfrak{g},[~,~],\alpha)$ be a regular hom-Lie algebra. Here, we need to take $\alpha$ invertible since we are going to use the coadjoint representation, defined in Example \ref{coadjoint rep}. It is known that $(\wedge^*\mathfrak{g},[~,~]_{g},\tilde\alpha)$ is a graded hom-Lie algebra, where the graded hom-Lie bracket is given by  
\begin{align*}
&[x_1\wedge\cdots\wedge x_n,y_1\wedge \cdots\wedge y_m]_{g}\\
=&\sum_{i=1}^n\sum_{j=1}^m (-1)^{i+j}[x_i,y_j]\wedge (\alpha(x_1)\wedge\cdots \widehat{\alpha(x_i)}\wedge\cdots\wedge \alpha(x_n)\wedge \alpha(y_1)\wedge \cdots \widehat{\alpha(y_j)}\wedge\cdots\wedge \alpha(y_m)),
\end{align*}
and the map $\tilde{\alpha}:\wedge^*\mathfrak{g}\rightarrow\wedge^*\mathfrak{g}$ is defined by 
$$\tilde{\alpha}(x_1\wedge x_2\wedge\cdots\wedge x_n)=\alpha(x_1)\wedge\alpha(x_2)\wedge\cdots\wedge \alpha(x_n),$$
for all $x_1,\cdots,x_n,y_1,\cdots,y_m\in \mathfrak{g}$.
Let us now consider a graded vector space $\mathcal{H}:=\oplus_{n\geq 1} \mathcal{H}^n,$ where 
$$\mathcal{H}^n:=\{\chi\in \wedge^n\mathfrak{g}~|~\tilde\alpha(\chi)=\chi\}.$$ 
If we restrict the bracket $[~,~]_g$ on $\mathcal{H}$, then we get a graded Lie algebra $(\mathcal{H},[~,~]_g)$.

A skew-symmetric $r$-matrix \cite{Yau1} on the hom-Lie algebra $(\mathfrak{g},[~,~],\alpha)$ is an element $r\in \mathcal{H}^2$, which satisfies the identity $[r,r]_g=0$.  In other words, $r$-matrices on hom-Lie algebra $(\mathfrak{g},[~,~],\alpha)$ are Maurer-Cartan elements of the associated graded Lie algebra $(\mathcal{H},[~,~]_g)$.

Any element $r\in \mathfrak{g}\otimes \mathfrak{g}$ corresponds to an operator $r^{\sharp}:\mathfrak{g}^*\rightarrow \mathfrak{g}$ and vice versa by the following expression
\begin{equation}\label{operator to r-matrix}
\langle\xi,r^{\sharp}(\eta)\rangle=\langle\xi\otimes \eta, r\rangle,\quad \mbox{for all }\xi,\eta\in \mathfrak{g}^*.
\end{equation}
The skew-symmetry of $r$ is equivalent to 
$$\langle\xi,r^{\sharp}(\eta)\rangle+\langle\eta,r^{\sharp}(\xi)\rangle=0.$$
Let $r\in \mathcal{H}^2$ and it is given by
\begin{equation}\label{def of r-matrix}
r=\sum_i(x_i\otimes y_i- y_i\otimes x_i),\quad\mbox{for some }x_i,y_i\in \mathfrak{g}.
\end{equation}
Then, the condition $[r,r]_{\mathfrak{g}}=0$ is equivalent to
$$[r^{12},r^{13}]+[r^{12},r^{23}]+[r^{13},r^{23}]=0,$$
where, 
\begin{align*}
[r^{12},r^{13}]&=\sum_i\sum_j[x_i,x_j]\otimes \tilde{y}_i\otimes\tilde{y}_j-[x_i,y_j]\otimes \tilde{y}_i\otimes\tilde{x}_j-[y_i,x_j]\otimes \tilde{x}_i\otimes\tilde{y}_j+[y_i,y_j]\otimes \tilde{x}_i\otimes\tilde{x}_j,\\
[r^{12},r^{13}]&=\sum_i\sum_j\tilde{x}_i\otimes[y_i,x_j]\otimes \tilde{y}_j-\tilde{x}_i\otimes[y_i,y_j]\otimes \tilde{x}_j-\tilde{y}_i\otimes[x_i,x_j]\otimes \tilde{y}_j+\tilde{y}_i\otimes[x_i,y_j]\otimes \tilde{x}_j,\\
[r^{13},r^{23}]&=\sum_i\sum_j\tilde{x}_i\otimes\tilde{x}_j\otimes [y_i,y_j]-\tilde{x}_i\otimes\tilde{y}_j\otimes [y_i,x_j]-\tilde{y}_i\otimes\tilde{x}_j\otimes[x_i,y_j] +\tilde{y}_i\otimes\tilde{y}_j\otimes [x_i,x_j],
\end{align*}
and $\tilde{x}:=\alpha(x)$. Moreover, the condition $\alpha^{\otimes 2}r=r$ implies that
\begin{align*}
\langle\xi,r^{\sharp}(\eta)\rangle=\langle\xi\otimes \eta, \alpha^{\otimes 2}r\rangle
=\langle\alpha^*(\xi)\otimes \alpha^*(\eta), r\rangle=\langle\alpha^*(\xi), r^{\sharp}(\alpha^*(\eta))\rangle=\langle\xi, \alpha\circ r^{\sharp}\circ\alpha^*(\eta)\rangle,
\end{align*}
for all $\xi,\eta\in \mathfrak{g}^*$. i.e., the operator $r^{\sharp}:\mathfrak{g}^*\rightarrow \mathfrak{g}$ satisfies 
\begin{equation}\label{cond1 for corresp}
\alpha\circ r^{\sharp}=r^{\sharp}\circ (\alpha^{-1})^*.
\end{equation} 

Next, we show that equation \eqref{operator to r-matrix} defines a bijective correspondence between $r$-matrices on a hom-Lie algebra $(\mathfrak{g},[~,~],\alpha)$  and  $\mathcal{O}$-operators on the hom-Lie algebra $(\mathfrak{g},[~,~],\alpha)$ with respect to the coadjoint representation $(\mathfrak{g}^*,(\alpha^{-1})^*,\rho^\star)$. We denote $\rho^\star(x)(\xi)$ simply by $\{x,\xi\}$, for all $x\in\mathfrak{g}$ and $\xi\in\mathfrak{g}^*$. 

For $\xi,\eta\in\mathfrak{g}^*,$ we get
\begin{align}\label{identity1}
&[r^{\sharp}(\alpha^*\xi),r^{\sharp}(\alpha^*\eta)]\\\nonumber
=&\sum_i\sum_j[\langle \alpha^*\xi,y_i\rangle x_i- \langle \alpha^*\xi,x_i\rangle y_i,\langle \alpha^*\eta,y_j\rangle x_j- \langle \alpha^*\eta,x_j\rangle y_j]\\\nonumber
=&\sum_{i}\sum_j \bigg(\langle \xi,\alpha(y_i)\rangle \langle \eta,\alpha(y_j)\rangle [x_i,x_j]-\langle \xi,\alpha(y_i)\rangle \langle \eta,\alpha(x_j)\rangle [x_i,y_j] \\\nonumber
&\quad\quad\quad\quad-\langle \xi,\alpha(x_i)\rangle \langle \eta,\alpha(y_j)\rangle [y_i,x_j]+\langle \xi,\alpha(x_i)\rangle \langle \eta,\alpha(x_j)\rangle [y_i,y_j]\bigg)\\\nonumber
=&-(\langle\xi,~\rangle\otimes\langle\eta,~\rangle\otimes \mathsf{Id}_\mathfrak{g})([r^{13},r^{23}]).
\end{align}
Note that by the equations \eqref{operator to r-matrix}, \eqref{def of r-matrix}, and the condition $\alpha^{\otimes 2}r=r$, we obtain
 $$r^{\sharp}(\eta)=\sum_i\langle \eta,y_i\rangle x_i- \langle \eta, x_i\rangle y_i=\sum_i\langle \eta,\alpha(y_i)\rangle \alpha(x_i)- \langle \eta, \alpha(x_i)\rangle \alpha(y_i).$$ 
Then, for all $\xi,\eta\in\mathfrak{g}^*$, we have the following identity    
\begin{align}\label{identity2}
&r^{\sharp}(\{r^{\sharp}(\alpha^*\xi),\alpha^*\eta\})\\\nonumber
=&r^{\sharp}\bigg(\sum_{i} \big(\langle \alpha^*\xi,y_i\rangle \{x_i,\alpha^*\eta\}- \langle \alpha^*\xi,x_i\rangle \{y_i,\alpha^*\eta\}\big)\bigg)\\\nonumber
%=&r^{\sharp}\bigg(\sum_{i} \big(\langle \xi,\alpha(y_i)\rangle \{x_i,\alpha^*\eta\}- \langle \xi,\alpha(x_i)\rangle \{y_i,\alpha^*\eta\}\big)\bigg)\\\nonumber
=&\sum_j\sum_i\bigg(\langle \xi,\alpha(y_i)\rangle \Big(\langle\{x_i,\alpha^*\eta\},\alpha(y_j)\rangle \alpha(x_j)-  \langle\{x_i,\alpha^*\eta\},\alpha(x_j)\rangle \alpha(y_j)\Big)\\\nonumber 
&\quad \quad \quad-\langle \xi,\alpha(x_i)\rangle \Big(\langle\{y_i,\alpha^*\eta\},\alpha(y_j)\rangle \alpha(x_j)-  \langle\{y_i,\alpha^*\eta\},\alpha(x_j)\rangle \alpha(y_j)\Big)\bigg)\\\nonumber
=&\sum_j\sum_i\bigg(\langle \xi,\alpha(y_i)\rangle \Big(\langle \eta,[x_i,y_j]\rangle \alpha(x_j)-  \langle\eta,[x_i,x_j]\rangle \alpha(y_j)\Big)\\\nonumber 
&\quad \quad \quad-\langle \xi,\alpha(x_i)\rangle \Big(\langle\eta,[y_i,y_j]\rangle \alpha(x_j)-  \langle\eta,[y_i,x_j]\rangle \alpha(y_j)\Big)\bigg)\\\nonumber
=&-(\langle\xi,~\rangle\otimes\langle\eta,~\rangle\otimes \mathsf{Id}_\mathfrak{g})([r^{12},r^{13}]).
\end{align}
Similarly,
\begin{align}\label{identity3}
r^{\sharp}(\{r^{\sharp}(\alpha^*\eta),\alpha^*\xi\})=(\langle\xi,~\rangle\otimes\langle\eta,~\rangle\otimes \mathsf{Id}_\mathfrak{g})([r^{12},r^{13}]).
\end{align}

Therefore, 
$$\big\langle \xi\otimes \eta\otimes \gamma,~[r,r]_{\mathfrak{g}}~\big\rangle= \Big\langle \gamma,~[r^{\sharp}(\alpha^*\xi),r^{\sharp}(\alpha^*\eta)]-r^{\sharp}(\{r^{\sharp}(\alpha^*\xi),\alpha^*\eta\}-\{r^{\sharp}(\alpha^*\eta),\alpha^*\xi\})\Big\rangle,$$
for all $\xi,\eta,\gamma\in\mathfrak{g}^*$. Hence, we have the following result.

\begin{theorem}\label{correspondence theorem}
Let $(\mathfrak{g},[~,~],\alpha)$ be a regular hom-Lie algebra. Then, an element $r\in \mathcal{H}^2$ is an $r$-matrix on the regular hom-Lie algebra $(\mathfrak{g},[~,~],\alpha)$ if and only if the associated operator $r^{\sharp}:\mathfrak{g}^*\rightarrow \mathfrak{g}$, defined by equation \eqref{operator to r-matrix}, is an $\mathcal{O}$-operator on hom-Lie algebra $(\mathfrak{g},[~,~],\alpha)$ with respect to the coadjoint representation $(\mathfrak{g}^*,(\alpha^{-1})^*,\rho^\star)$.
\end{theorem}

\begin{remark}
There is another version of $r$-matrix defined in \cite{hom-Liebi}. Here, we followed the version defined in \cite{Yau1} to derive a connection between $\mathcal{O}$-operators and $r$-matrices on regular hom-Lie algebras (Theorem \ref{correspondence theorem}). Also, we refer to \cite{Yau1} for more details on hom-Yang-Baxter equations.
\end{remark}

If $r$ is an $r$-matrix on the regular hom-Lie algebra $(\mathfrak{g},[~,~],\alpha)$, then from Theorem \ref{correspondence theorem} $r^\sharp:\mathfrak{g}^*\rightarrow \mathfrak{g}$ is an $\mathcal{O}$-operator on hom-Lie algebra $(\mathfrak{g},[~,~],\alpha)$ with respect to the coadjoint representation $(\mathfrak{g}^*,(\alpha^{-1})^*,\rho^\star)$. Thus, from Proposition \ref{induced hom-pre-Lie algebra}, the triplet $(\mathfrak{g}^*,[~,~]_r,(\alpha^{-1})^*)$ is the induced sub-adjacent hom-Lie algebra with the bracket 
$$[\xi,\eta]_{r}:=\mathsf{ad}^*_{r^{\sharp}(\xi)}\eta-\mathsf{ad}^*_{r^{\sharp}(\eta)}(\xi)=\{r^{\sharp}(\eta),\xi\}-\{r^{\sharp}(\xi),\eta\},\quad\mbox{for all }\xi,\eta\in\mathfrak{g}^*.$$

Moreover, from Proposition \ref{rep associated to O-operator}, it follows that $(\mathfrak{g},\alpha,\rho_{r})$ is a representation of the hom-Lie algebra $(\mathfrak{g}^*,[~,~]_r,(\alpha^{-1})^*)$, where the map $\rho_{r}:\mathfrak{g}^*\rightarrow\mathsf{End}(\mathfrak{g})$ is given by 
$$\rho_{r}(\xi)(x)=[r^{\sharp}(\xi),x]+r^{\sharp} \{x,\xi\}\quad\mbox{for all }\xi\in \mathfrak{g}^*,~x\in \mathfrak{g}.$$
It is also clear that the induced representation $\rho_r:\mathfrak{g}^*\rightarrow \mathsf{End}(\mathfrak{g})$ is the same as the coadjoint representation $\mathsf{ad}^*:\mathfrak{g}^*\rightarrow \mathsf{End}(\mathfrak{g})$ of the hom-Lie algebra $(\mathfrak{g}^*,[~,~]_r,(\alpha^{-1})^*)$ on the pair $(\mathfrak{g},\alpha)$. In fact,  
\begin{align*}
\langle\eta, \rho_r(\xi)(x)\rangle&=\langle\eta, r^{\sharp} \{x,\xi\}+[r^{\sharp}(\xi),x]\rangle\\&=-\langle \{x,\xi\},r^{\sharp}(\eta)\rangle - \langle\eta, [x,r^{\sharp}(\xi)]\rangle\\
&=-\langle\{x,\xi\},\alpha(r^{\sharp}(\alpha^*\eta))\rangle- \langle\eta,[x,r^{\sharp}(\xi)]\rangle\\
&=\big\langle\xi,~\alpha^{-1}[x,r^{\sharp}(\alpha^*\eta)]~\big\rangle- \langle\eta, [x,r^{\sharp}(\xi)]\rangle\\
&=\langle \{r^{\sharp}(\alpha^*\eta),\xi\}-\{r^{\sharp}(\xi),\alpha^*\eta\},\alpha(x)\rangle\\
&=\langle [(\alpha^*)^2(\eta),\alpha^*\xi]_r,x\rangle\\
&=\langle \eta,\mathsf{ad}^*_{\xi}(x)\rangle, \quad \quad \mbox{for all } x\in \mathfrak{g},~\xi,\eta\in \mathfrak{g}^*.
\end{align*}

In \cite{Sheng3}, the authors define the notion of a weak morphism of $r$-matrices. Now, we extend this notion to $r$-matrices on hom-Lie algebras.      

\begin{definition}
Let $r_1,r_2$ be skew-symmetric $r$-matrices on a hom-Lie algebra $(\mathfrak{g},[~,~],\alpha)$. A pair $(\phi,\psi)$, consisting of a hom-Lie algebra homomorphism $\phi:\mathfrak{g}\rightarrow \mathfrak{g}$ and a linear map $\psi:\mathfrak{g}\rightarrow \mathfrak{g}$, is said to be a weak homomorphism from $r_1$ to $r_2$ if the following conditions hold
\begin{enumerate}[label=\alph*)]
\item $\psi\circ \alpha=\alpha\circ \psi,$
\item $(\psi\otimes \mathsf{Id}_{\mathfrak{g}})(r_1)=(\mathsf{Id}_{\mathfrak{g}}\otimes \phi)(r_2),$
\item $\psi([\phi(x),y])=[x,\psi(y)],~~~\mbox{for all }x,y\in \mathfrak{g}$.
\end{enumerate} 
A weak homomorphism $(\phi,\psi)$ is called a `weak isomorphism' if $\phi$ and $\psi$ are linear automorphisms on $\mathfrak{g}$. 
\end{definition}
The following result establishes the relationship between weak homomorphisms between $r$-matrices on hom-Lie algebras and homomorphisms between the corresponding $\mathcal{O}$-operators on hom-Lie algebras with respect to the coadjoint representation.

\begin{proposition}
Let $r_1,r_2$ be skew-symmetric $r$-matrices on a hom-Lie algebra $(\mathfrak{g},[~,~],\alpha)$. A pair $(\phi,\psi)$ is a weak homomorphism from $r_1$ to $r_2$ if and only if $(\phi,\psi^*)$ is a homomorphism of $\mathcal{O}$-operators from $r_1^\sharp$ to $r_2^\sharp$.
\end{proposition}
\begin{proof} 
First, let us observe that any skew-symmetric $r$-matrix 
$r\in \mathcal{H}^2$, given by $$r=\sum_{i}(a_i\otimes b_i- b_i\otimes a_i)\quad\mbox{for some } a_i, b_i\in \mathfrak{g},$$ can also be written as a sum $r=\sum_{j}(a_j^\prime\otimes b_j^\prime)$ by changing the indexing. Let $r_1=\sum_i x_i\otimes y_i$ and $r_2=\sum_i x^{\prime}_i\otimes y^{\prime}_i$ be skew-symmetric $r$-matrices. The associated $\mathcal{O}$-operators are given by 
$${r_1}^\sharp=\sum\limits_{i}\langle\xi,x_i\rangle y_i \quad \mbox{and}\quad r_2^\sharp=\sum\limits_{j}\langle\xi,x^{\prime}_j\rangle y^{\prime}_j, \quad\mbox{for }\xi\in \mathfrak{g}^*.$$  

Let the pair $(\phi,\psi)$ be a weak homomorphism from $r_1$ to $r_2$. Then, the map $\phi:\mathfrak{g}\rightarrow \mathfrak{g}$ is a hom-Lie algebra homomorphism, $\psi\circ\alpha=\alpha\circ \psi$, and the following conditions are satisfied
\begin{equation}\label{hom:con1}
(\psi\otimes \mathsf{Id}_{\mathfrak{g}})(r_1)=(\mathsf{Id}_{\mathfrak{g}}\otimes \psi)(r_2), 
\end{equation}
\begin{equation}\label{hom:con2}
\psi([\phi(x),y])=[x,\psi(y)],\quad\mbox{for all }x,y\in \mathfrak{g}.
\end{equation}
Let us note that $\psi\circ\alpha=\alpha\circ \psi$ if and only if $\psi^*\circ(\alpha^{-1})^*=(\alpha^{-1})^*\circ \psi^*$. Next, we show that the conditions \eqref{hom:con1} and \eqref{hom:con2} hold if and only if 
the following conditions \eqref{hom:con3} and \eqref{hom:con4} hold true, respectively.
\begin{equation}\label{hom:con3}
r_2^\sharp\circ \psi^*=\phi\circ r_1^\sharp,
\end{equation}
\begin{equation}\label{hom:con4}
\psi^*(\{x,\xi\})=\{\phi(x),\psi^*(\xi)\}, ~~~\mbox{for all }x,y\in \mathfrak{g}.
\end{equation}

In order to show the desired result, let us consider the following expressions 
\begin{equation*}\label{condition-5.9}
\langle\xi\otimes\eta,(\mathsf{Id}_{\mathfrak{g}}\otimes \phi)(r_1)\rangle=\sum\limits_{i}\langle \xi,x_i\rangle \langle \eta,\phi(y_i)\rangle =\langle\eta,\phi(\sum\limits_{i}\langle\xi, x_i \rangle y_i\rangle=\langle\eta,\phi r_1^\sharp(\xi)\rangle
\end{equation*}
and
\begin{equation*}\label{condition-5.10}
\langle\xi\otimes \eta,(\psi\otimes \mathsf{Id}_\mathfrak{g})(r_2)\rangle=\sum\limits_{j}\langle\xi, \psi(x_j^\prime)\rangle \langle \eta,y_j^\prime\rangle =\langle\eta,\sum\limits_j\langle \psi^*\xi,x_j^\prime \rangle y_j^\prime\rangle=\langle\eta,r_2^\sharp(\psi^*\xi)\rangle.
\end{equation*}
Therefore, the condition \eqref{hom:con1} holds if and only if the condition \eqref{hom:con3} holds true. Moreover, the identities $\psi\circ \alpha=\alpha\circ \psi$ and $\phi\circ\alpha=\alpha\circ \phi$ implies that  
\begin{equation*}\label{condition-5.7}
\langle\xi,\psi[y,\phi(x)]\rangle=\Big\langle\psi^*(\xi),[y,\phi(x)]\Big\rangle=\langle\{\alpha\phi(x),\psi^*(\xi)\},\alpha^2(y)\rangle=\langle\{\phi(\alpha (x)),\psi^*(\xi)\},\alpha^2(y)\rangle,
\end{equation*}
\begin{equation*}\label{condition-5.8}
\langle \xi,[\psi(y),x]\rangle=\langle\{\alpha(x),\xi\},\alpha^2(\psi(y))\rangle=\langle\{\alpha(x),\xi\},\psi\alpha^2(y)\rangle=\langle\psi^*\{\alpha(x),\xi\},\alpha^2(y)\rangle.
\end{equation*}
Since $\alpha$ is an automorphism, it follows that the condition \eqref{hom:con2} holds if and only if the condition \eqref{hom:con4} holds true.

In other words, we can conclude that $(\phi,\psi)$ is a weak homomorphism from $r_1$ to $r_2$ if and only if $\phi$ is a hom-Lie algebra homomorphism, $\psi^*\circ(\alpha^{-1})^*=(\alpha^{-1})^*\circ \psi^*$, and the conditions \eqref{hom:con3}, \eqref{hom:con4} are satisfied. Equivalently, from the definition \ref{morphism of O-operators}, it follows that $(\phi,\psi)$ is a weak homomorphism from $r_1$ to $r_2$ if and only if the pair $(\phi,\psi^*)$ is a morphism of $\mathcal{O}$-operators from $r_1^\sharp$ to $r_2^\sharp$.
\end{proof}

\begin{definition}
Let $r$ be a skew-symmetric $r$-matrix on a regular hom-Lie algebra $(\mathfrak{g},[~,~],\alpha)$. Let us consider a linear sum $r_t:=r+t\tau\in \mathcal{H}^2$ for some $\tau\in \mathcal{H}^2$. If $r_t$ is a skew-symmetric $r$-matrix on the hom-Lie algebra $(\mathfrak{g},[~,~],\alpha)$, then it is called a linear deformation generated by the element $\tau\in \mathcal{H}^2$.  
\end{definition}

\begin{definition}
Let $r$ be a skew-symmetric $r$-matrix on a regular hom-Lie algebra $(\mathfrak{g},[~,~],\alpha)$. Then, linear deformations $r^1_t:=r+t\tau_1$ and $r^2_t:=r+t\tau_2$, generated by elements $\tau_1$ and $\tau_2$ in $\mathcal{H}^2$ are equivalent if there exists an element $x\in \mathfrak{g}$ satisfying $\alpha(x)=x$ such that the pair $(\mathsf{Id}+t \mathsf{ad}_x^{\dagger},\mathsf{Id}-t \mathsf{ad}_x^{\dagger})$ is a weak homomorphism from $r^2_t$ to $r^1_t$.
\end{definition}

\begin{proposition}
Let $(\mathfrak{g},[~,~],\alpha)$ be a hom-Lie algebra and $r\in \mathcal{H}^2$ be a skew-symmetric $r$-matrix on $(\mathfrak{g},[~,~],\alpha)$. Then,
\begin{enumerate}
\item An element $\tau\in \mathcal{H}^2$ generates a linear deformation of $r$ if and only if the induced map $\tau^\sharp:\mathfrak{g}^*\rightarrow\mathfrak{g}$ generates a linear deformation of the $\mathcal{O}$-operator $r^{\sharp}$.
\item Linear deformations $r^1_t:=r+t\tau_1$ and $r^2_t:=r+t\tau_2$ are equivalent if and only if the linear deformations $(r^1_t)^\sharp$ and $(r^2_t)^\sharp$ of the $\mathcal{O}$-operator $r^{\sharp}$ are equivalent.
\end{enumerate}
\end{proposition}

\subsection*{Formal deformations of $r$-matrices}
Let $r\in \mathcal{H}^2$ be a skew-symmetric $r$-matrix on a hom-Lie algebra $(\mathfrak{g},[~,~],\alpha)$. A formal sum $\textstyle{r_t:=r+\sum_{i\geq 1} t^i r_i}$ is called a formal deformation of $r$ if it satisfies 
$$[[r_t,r_t]]_{\mathfrak{g}_t}=0.$$
Here, $r_t\in \wedge^2\mathfrak{g}[[t]]$ and the bracket $[[~,~]]_{\mathfrak{g}_t}$ is the Lie bracket on the exterior algebra $\wedge^*\mathfrak{g}[[t]]$. 
Next, it easily follows that 
\begin{enumerate}
\item A formal sum $\textstyle{r_t:=r+\sum_{i\geq 1} t^i r_i}$ is a formal deformation of $r$ if and only if the induced map $r_t^{\sharp}:=r^{\sharp}+\sum_{i\geq 1} t^i r_i^\sharp$ is a formal deformation of the $\mathcal{O}$-operator $r^{\sharp}$.

\item The formal deformations $r^1_t:=r+\sum_{i\geq 1} t^i r^1_i$ and $r^2_t:=r+\sum_{i\geq 1} t^i r^2_i$ are equivalent if and only if the formal deformations $(r^1_t)^{\sharp}$ and $(r^1_t)^\sharp$ of the $\mathcal{O}$-operator $r^{\sharp}$ are equivalent.

\end{enumerate}

\section*{Conclusion}
In this work, we defined a differential graded Lie algebra whose Maurer-Cartan elements characterize $\mathcal{O}$-operators on hom-Lie algebras with respect to a representation. This characterization helps us to discuss formal deformation theory for 
an $\mathcal{O}$-operator in the regular case. However, with this approach, we can only study deformations of the operator, but the study of simultaneous deformations of hom-Lie algebra with a representation and the $\mathcal{O}$-operator is not possible. 

To study such simultaneous deformations, we define a relative Rota-Baxter hom-Lie algebra as a triple $((L,[~,~],\alpha),(V,\rho,\beta),T)$ consisting of a hom-Lie algebra $(L,[~,~],\alpha)$ with a representation $(V,\rho,\beta)$ and an $\mathcal{O}$-operator $T:V\rightarrow L$ on $(L,[~,~],\alpha)$ with respect to $(V,\rho,\beta)$. 
In our subsequent paper, we follow a more general approach of the recent work \cite{laza} to study deformations of relative Rota-Baxter hom-Lie algebras via $L_{\infty}$-algebras. In particular, we use the higher derived bracket construction of Voronov \cite{Voronov} to obtain an $L_{\infty}$-algebra starting from the graded Lie algebra obtained in this paper. The Maurer-Cartan elements of this $L_{\infty}$-algebra characterize relative Rota-Baxter hom-Lie algebras. Maurer-Cartan characterization of relative Rota-Baxter hom-Lie algebras also leads to the homotopy Rota-Baxter operators on $HL_{\infty}$-algebras \cite{HL-infty}. 
%\subsection{Quasi-triangular hom-Lie bialgebras}
%Let us recall that a hom-Lie bialgebra is a hom-Lie algebra $(\mathfrak{g},[~,~],\alpha)$ that is also equipped with a Lie coalgebra structure $\nabla:\mathfrak{g}\rightarrow \wedge^2\mathfrak{g}$ such that $\nabla$ is a $1$-cocycle on $\mathfrak{g}$ with coefficients in $\wedge^2\mathfrak{g}$.

\end{document}